\documentclass{compositio}
\usepackage{amssymb, amsmath, latexsym, amsthm, amscd, graphicx, xcolor, hyperref, bm,enumitem, soul}
\usepackage{nicefrac}
\usepackage{soul}

\usepackage[T2A]{fontenc}
\usepackage[utf8]{inputenc}

\newcommand{\ignore}[1]{}

\def\R{\mathbb R}
\def\N{\mathbb N}

\def\Z{\mathbb Z}
\def\Q{\mathbb Q}

\def\I{\mathcal I}

\def\A{\mathcal A}

\def\={\equiv}
\def\<{\langle}
\def\>{\rangle}

\def\Sing{\operatorname{Sing}}
\newcommand{\UA}{\operatorname{U\mspace{-3mu}A}}
\def\rank{\operatorname{rank}}
\def\ov{\overline}

\def\r{\mathbf{r}}
\def\v{\mathbf{v}}

\def\x{\mathbf{x}}
\def\y{\mathbf{y}}
\def\z{\mathbf{z}}
\def\q{\mathbf{q}}
\def\a{\mathbf{a}}

\def\1{\mathbf{1}}

\def\Id{\operatorname{Id}}

\def\Gr{\operatorname{Gr}}
\def\dist{\operatorname{dist}}

\def\opt1{\operatorname{T}^1}

\def\GL{\operatorname{GL}}

\def\b{\mathbf{b}}

\def\e{\mathbf{e}}

\def\c{\mathbf{c}}

\def\x{\mathbf{x}}

\newcommand{\va}{{\boldsymbol{\alpha}}}
\newcommand{\vb}{{\boldsymbol{\beta}}}
\newcommand{\vw}{{\boldsymbol{\omega}}}

\newcommand{\vp}{{\bf p}}
\newcommand{\vq}{{\bf q}}

\newcommand{\nz}{\smallsetminus\{0\}}

\newcommand\da{Diophantine approximation}
\newcommand\hd{Hausdorff dimension}

\newcommand\amr{$A\in M_{m,n}(\R)
$}
\newcommand{\df}{{\, \stackrel{\mathrm{def}}{=}\, }}

\newcommand\mr{M_{m,n}(\R)
}

\numberwithin{equation}{section}

\newif\ifdraft\drafttrue

\draftfalse

\newcommand\eq[2]{{\ifdraft{\ \tt [#1]}\else\ignorespaces\fi}\begin{equation}\label{eq:#1}{#2}\end{equation}}
\newcommand {\equ}[1]     {\eqref{eq:#1}}

\newcommand {\comm}[1]   {\textcolor{red}{#1}}

\newcommand{\vre}{\varepsilon}
\newcommand{\sm}{\smallsetminus}

\newtheorem{theorem}{Theorem}[section]

\newtheorem{proposition}[theorem]{Proposition}
\newtheorem{claim}[theorem]{Claim}

\newtheorem{corollary}[theorem]{Corollary}
\newtheorem{lemma}[theorem]{Lemma}
\newtheorem{sublemma}[theorem]{Sublemma}
\newtheorem{definition}[theorem]{Definition}

\newtheorem{remark}[theorem]{Remark}

\title[Uniform approximation: weights, intersections, rates]{Singularity, weighted uniform
  approximation, intersections and rates}
\author[Kleinbock]{Dmitry Kleinbock}
\email{kleinboc@brandeis.edu}
\address{Department of Mathematics, Brandeis University, Waltham MA, USA}
\author[Moshchevitin]{Nikolay Moshchevitin}
\email{nikolai.moshchevitin@tuwien.ac.at}
\address{Institut f\"ur diskrete Mathematik und Geometrie, Technische Universit\"at Wien, Freihaus, Wiedner Hauptstraße 8, A-1040, Wien}
\author[Warren]{Jacqueline M.\ Warren}
\email{j4warren@ucsd.edu}
\address{Department of Mathematics, University of California San Diego,  CA, USA}
\author[Weiss]{Barak Weiss}
\email{barakw@tauex.tau.ac.il}
\address{Department of Mathematics, Tel Aviv University, Tel Aviv, Israel}
\shortauthors{Kleinbock, Moshchevitin, Warren and Weiss}

\subjclass{11H06, 11J13, 11J54, 11J82}
\keywords{uniform Diophantine approximation, singular systems of linear forms, \da\ with weights, \da\  on manifolds and fractals}
\thanks{The authors are grateful for funding by the National Science Foundation
under grant DMS-2155111,   {Austrian Science Fund (FWF), Forschungsprojekt PAT1961524},  the Israel Science Foundation under
grants ISF 2019/19 and ISF-NSFC 3739/21, 
and the Zuckerman STEM Leadership
  Program.}

\date{March 2025}

\begin{document}

\begin{abstract}
A classical argument was introduced by Khintchine in 1926 in order to
exhibit the existence of 
totally irrational singular linear forms in two variables. This
argument was subsequently revisited and extended by many authors. For
instance, in {1959} Jarn\'{\i}k used it to show that for $n \geq 2$ and for {any
non-increasing positive  $f$} there are totally
irrational matrices $A \in M_{m,n}(\R)$ such that for all large enough
$t$ there are $\mathbf{p} \in \Z^m, \mathbf{q} \in \Z^n \sm \{0\}$
with
$$\|\mathbf{q}\| \leq t \ \text{ and } \ \|A \mathbf{q} - \mathbf{p}\| \leq
f(t).$$
We denote the collection of such matrices by
$\UA^*_{m,n}(f)$. We adapt Khintchine's argument to show that the
sets $\UA^*_{m,n}(f)$, and their weighted analogues
$\UA^*_{m,n}(f, \vw)$, intersect many manifolds and fractals,
and have strong intersection properties. For example, we show that:
\begin{itemize}
  \item
  When $n \geq 2$, the set $\bigcap_{\vw} \UA^*(f, \vw) $,
  where the intersection is over all weights $\vw$, is nonempty,
  and moreover intersects many manifolds and fractals;
\item
  For $n \geq 2$, there are vectors in $\R^n$  which are
  simultaneously $k$-singular for every 
  $k$, in the sense of Yu;

  \item
when $n \geq 3$, $\UA^*_{1,n}(f) + \UA^*_{1,n}(f)  =
\R^n$.
\end{itemize}
We also obtain new bounds on the rate of singularity which can be
attained by  {column vectors in analytic submanifolds  of dimension at least 2} in $\R^n$. 

  \end{abstract}
  
\maketitle



\section{Introduction}


\subsection{Background and history}\label{bh} In 1926, Khintchine \cite{H1} introduced the property of singularity of
vectors,  which later was extended to systems of linear forms by Jarn\'{\i}k \cite{Jarnik}: 
A
matrix $A \in M_{m,n}(\R)$ (viewed as a system of $m$ linear forms in $n$ variables) 
is {\em singular} {(notation: $A\in\Sing_{m,n}$)} if for any
$\vre>0$ there is $t_0>0$ such that for all $t \geq t_0$ there exist
$\mathbf{q} \in \Z^n \sm \{0\}$ and $\mathbf{p}  \in \Z^m$ such that 
\begin{equation}\label{eq: singularity}
  \|\mathbf{q}\|  \leq t \ \ \text{ and } \ \ 
  \|A\mathbf{q}  - \mathbf{p} \| \le \frac{\vre}{t^{n/m}}. 
\end{equation}
{Here $M_{m,n}(\R)$ 
{is} the collection of $m\times n$ real
  matrices, and $\|
\cdot \|$  stands for the supremum norm on $\R^m$ and $\R^n$, although
the choice of the norm will be immaterial for this work.}

Note that \eqref{eq: singularity} always has nontrivial integer solutions if $\vre = 1$ (Dirichlet's Theorem).
Khintchine, and then  Jarn\'{\i}k in bigger generality, showed
that while for $m=n=1$ the only
singular numbers are rationals, for $\max(m,n) > 1$ there exist
uncountably many singular matrices which are  {\sl totally
irrational}, that is, not contained in any rational affine subspace of $\mr$.
 {We will denote by  $\Sing^*_{m,n}$ the set of totally irrational singular $m\times n$ matrices. Since the work of Khintchine} a
large  
body of work in Diophantine approximation has been devoted to the
understanding of 
{$\Sing^*_{m,n}$ and related sets.  See
\cite{Moshchevitin_survey} for a detailed survey and an extensive
bibliography, {and \cite{KKLM, dfsu, BCC} for the computation of the \hd\
  of $\Sing_{m,n}$ and for a discussion of related work. 

\ignore{According to his definition, 
a vector $\mathbf{x} \in \R^d$ is {\em singular} if for any
$\vre>0$ there is $t_0>0$ such that for all $t \geq t_0$, there exist
$\mathbf{q} \in \Z^d \sm \{0\}$ and $p \in \Z$ such that 
\begin{equation}\label{eq: singularity}
  \|\mathbf{q}\|  \le t \ \ \text{ and } \ \ 
  |\mathbf{q} \cdot \mathbf{x} - p| < \frac{\vre}{t^d}. 
\end{equation}
Khintchine showed
that while for $d=1$, the only
singular numbers are rationals, for $d>1$ there are totally
irrational singular vectors.\footnote{Khintchine only dealt with the
  case $d=2$ in his original paper, and the general case was
  documented in 1959 by Jarn\'{\i}k.} Following Khintchine's discovery, a
large  
body of work in Diophantine approximations was devoted to the
understanding of singular vectors, and this has spawned the topic of
uniform Diophantine approximation. See
\cite{Moshchevitin_survey} for a detailed survey and an extensive bibliography.}

\smallskip
	{A natural way to generalize \eqref{eq: singularity} is to replace $\frac{\vre}{t^{n/m}}$ in the right hand side by an arbitrary approximating function.}

\begin{definition}\label{Dir} \rm
Given a  
function $f: \R_{>0} \to \R_{>0}$, say 
that  \amr\ is {\sl uniformly $f$-approximable}\footnote{This
  property, {with a slightly different parametrization,} has also been called `$f$-Dirichlet' and `$f$-singular' {in the literature,
  see e.g.\ \cite{kw1} and \cite{dfsu}}.}, or {\sl $f$-uniform} for brevity,
if 
for every sufficiently 
large $t>0$  one can find 
$\vq  \in \Z^n\nz$ and $\vp \in \Z^m$
 with 
\eq{digeneral}{
{\|A\vq - \vp \|
\le  f(t)}
  \ \ \ \mathrm{and}  
\ \ \|\vq\|
\le  t
.} 
Denote by   $\UA_{m,n}(f)$ the set of 
{ $f$-uniform}
$m\times n$ matrices, and by  $\UA^*_{m,n}(f)$ the set of totally
irrational $A\in \UA_{m,n}(f)$.
\end{definition}

{We remark that in general the field} of
{\sl uniform Diophantine approximation} {deals with the solvability of systems of type \eqref{eq: singularity} or \equ{digeneral} for \textit{all large enough} values of $t$, as opposed to the (much better understood) theory of {\sl asymptotic  approximation}. There one can define the sets  $\A_{m,n}(f)$ 
 of 
 $f$-approximable
$m\times n$ matrices, where those systems are required to have integer solutions for \textit{an unbounded set} of $t>0$}.
\smallskip

{Dirichlet's Theorem   asserts that $
  \UA_{m,n}(\phi_{n/m}) = \mr$, where we use the notation $\phi_a(t)
  \df t^{-a}$. Also one clearly has
 $$\Sing_{m,n} = \bigcap _{\varepsilon > 0}\UA_{m,n}(\varepsilon
 \phi_{n/m}),$$
 and, more generally, $\UA_{m,n}(f)\subset \Sing_{m,n}$ 
as long as $f$
satisfies
\begin{equation}\label{eq: f as in}
  t^{n/m}f(t) \to_{t
  \to \infty}
0.
\end{equation}}

It is well-known (see \cite[Ch.~V, \S 7]{Cassels_book})
that
the set of singular 
matrices is a nullset in $M_{m,n}(\R)$, and hence so is
the set of $f$-uniform matrices for $f$ satisfying \eqref{eq: f as
  in}. 
 It is also clear that matrices $A$ such that $  A\mathbf{q}  = \mathbf{p}$
   for some $\mathbf{q} \in \Z^n \sm \{0\}$ and $\mathbf{p}  \in \Z^m$
   are in $\UA_{m,n}(f)$ for any 
   $f$. On the other hand we
   have:

 \begin{theorem}[Jarn\'{\i}k \cite{Jarnik}] \label{thm: classical theorem}
	Let $m,n\in \N$. 
	{\begin{itemize}
	\item[\rm (a)] If $n > 1$, then for any 
        non-increasing 
$f: \R_{>0} \to \R_{>0}$ the set of totally irrational
     $f$-uniform $m \times n$ matrices	is
     uncountable and dense. 
    \item[\rm (b)]  If, in addition, one assumes that \eq{vectorbound}{\lim_{t\to\infty}tf(t) = \infty,} then for any $m>1$ the set of totally irrational
     $f$-uniform {$m \times 1$} matrices (column vectors)	is
     uncountable and dense. 
     \end{itemize}} \end{theorem}

 


\begin{remark}\label{exponents}\rm {The above theorem  
is often stated in a 
simplified form} by means of introducing the notion of {the} \textsl{uniform Diophantine exponent} $\hat \omega(A)$  to be the supremum of $a$
so that $A$ is 
$\phi_a$-uniform.
In other words, 
\eq{omegahat}{\hat{\omega}(A) \df \sup
  \left\{a> 0 \left| 
 \begin{aligned}
 \text{ for all large enough }
  t >0\qquad\\
  \ \exists \,\vq \in \Z^n\nz,\ \mathbf{p} \in \Z^m
  \qquad \\
   \text{
    such that }\equ{digeneral} \text{ holds for 
    $f=\phi_a$}\end{aligned}\right.\right\}. 
}
Then Theorem
 \ref{thm: classical theorem} 
asserts that there exist uncountably many {totally irrational} $m\times n$ matrices $A$ with $\hat \omega(A) = \infty$ whenever $n>1$, and  uncountably many {totally irrational} $m\times 1$ matrices (column vectors) $\x$ with $\hat \omega(\x) = 1$ {as long as $m>1$}.

{Note that in the column vector case the extra condition \equ{vectorbound} cannot be removed: indeed, if  $\x$ is such that $\hat \omega(\x) > 1$, then one has  $\hat \omega(x_i) > 1$ for every component $x_i$ of $\x$, which {necessarily} implies that $x_i\in\Q$ for all $i$. See also 
a discussion in \cite[\S3.2]{dfsu}.}

\end{remark}

 The {argument} of Khintchine and Jarn\'{\i}k has been utilized and extended by many people, e.g.\ {\cite{dani, Barak_GAFA, KWe2}}.
 Most recently, in \cite{KMW} three of the authors of the
present paper in the case {$m = 1$} showed that 
{there exist uncountably many $f$-uniform {row} vectors on} certain submanifolds
and fractals. As a special case, they showed that if {$Y$} is either
a 
{connected analytic submanifold of $\R^n$ of  dimension at least 2 not contained in a rational affine hyperplane}, or   a product $\mathcal{C}
\times \cdots \times \mathcal{C}$ of $n \geq 2$ copies of the
{middle-third} Cantor set $\mathcal{C}$, then
{$Y$} contains uncountably many $f$-uniform   vectors for any {approximating function} $f$.

\subsection{Diophantine systems}\label{ds} The goal of this paper is to strengthen {the aforementioned} results, 
namely provide conditions on countably many maps {$\{\varphi_k : k\in\N\}$} from some metric space $Y$  to $M_{m_k,n_k}(\R)$ under which 
there is an uncountable and dense set of $y\in Y$ for which 
$\varphi_k(y)$ are $f$-uniform for every $k$. All this will follow from an
abstract result (Theorem \ref{thm: special axiomatic theorem}) which
will generalize many constructions of {singular/uniform} objects  existing in
the literature.  The proof of Theorem \ref{thm: special axiomatic
  theorem} utilizes once again the original strategy introduced by Khintchine.

\smallskip

In order to state our construction in full generality we
 introduce the notion of   a {\sl Diophantine system}, which is
 a {triple $\mathcal{X} = \left( X, \mathcal{D}, \mathcal{H}\right)$,
  where:
  \begin{itemize}
  \item[$\bullet$] $X$ is a topological space;
  \item[$\bullet$] 
  $\mathcal{D}= \{d_s : s\in  \I\}$ is a sequence of continuous functions on $X$ taking nonnegative values ({\sl generalized distance functions});

 \item[$\bullet$] 
    $\mathcal{H}= \{h_s: s\in  \I\}$ is a sequence of positive real numbers; we will say that $h_s$ is the {\sl height} associated with 
    $d_s$.
   \end{itemize}
  Here $\I$ is a fixed countable set\footnote{{For our exposition it will be convenient to have a freedom of choice of  $\I$ instead of always using $\I = \N$.}}  of indices.  
    This generalizes the definition of a {\sl Diophantine space}
    originally introduced in \cite{fsy}; namely, 
  by a
    Diophantine space one means a triple $(X, \mathcal{R}, H)$, where
    $X$ is a complete metric space, $\mathcal{R}$ is a {countable} 
dense subset of $X$, and $H : \mathcal{R} \to (0, \infty)$ gives the
heights of points in $\mathcal{R}$.  {To view it as a Diophantine
  system as in our definition one can simply enumerate the elements of
  $\mathcal{R}  = \{r_s: s\in  \I\}$ and let $d_s(x) \df \dist(x,r_s)$
  and  $h_s \df H(r_s)$. Similarly one can treat a more general
  situation when $\mathcal{R}$ is a countable collection of closed
  subsets of $X$ ({\sl resonant sets}), which in our notation are
  simply the zero sets of the generalized distance functions $d_s$. Our
  setup  however {makes it possible} to define the distance from each
  individual set in a different way, as well as count the same set
  several (perhaps infinitely many) times.} 

   Now suppose we are given a Diophantine system {$\mathcal{X} =
     \left( X, \mathcal{D}, \mathcal{H}\right)$. Then we can introduce
     the notion of $f$-uniform points for any function $f: \R_{>0} \to
     \R_{>0}$ (the {\sl approximating function}):
     
     \begin{definition}\label{DirAbstract} \rm
$x\in X$
     is {\sl $(\mathcal{X},f)$-uniform}   if 
for every sufficiently
large $t$  one can find $s\in\I$ such that 
$${
d_s (x)
\le f(t)
  \ \ \ \mathrm{and}  
\ \ h_s
\le  t
.}$$
We will denote by  $\UA_{\mathcal{X}}(f)$ the set of $(\mathcal{X},f)$-uniform points of $X$. 
\end{definition}


As the reader can easily check, the following choices define a
Diophantine  {system}, which we refer to as the {\sl standard Diophantine
  {system}} {$\mathcal{X}_{m,n}$}:
 {\eq{classical}{
 \begin{aligned}X = \mr,\  \I =  \Z^m\times (\Z^n\nz),\\ 
h_{\vp,\vq} = \|\vq\|, \ d_{\vp,\vq}(A) = \|A\vq - \vp\|{.}\quad\end{aligned}}
It is clear that 
 a matrix
is $({\mathcal{X}_{m,n}}, f)$-uniform if and only if it 
{is} $f$-uniform 
in the sense of Definition~\ref{Dir}. 

\smallskip
Trivially 
every point in the union of resonant sets $d_s^{-1}(\{0\})$ is   $(\mathcal{X},f)$-uniform for any choice of $\mathcal{H}$  and $f$. {In the next section we will state our} main theorem, 
{which} provides a set of conditions guaranteeing the existence of uncountably many non-trivial $(\mathcal{X},f)$-uniform points. {More generally, it involves an auxiliary metric space $Y$ and countably many maps from $Y$ to possibly different metric spaces $X_i$, each endowed with its own Diophantine system.}

\subsection{The main result}
\label{result}
{To state our main abstract theorem, it will be convenient  to introduce some more terminology. Let $Y$ be a 
	metric space and let $\mathcal{L}$, $\mathcal{R}$ be two  collections of proper closed subsets of 
$Y$.  Say that    $\mathcal{L}$   is 
{\sl   totally
dense  relative to $\mathcal{R}$} 
 if 	\eq{vd1}{\bigcup_{L\in \mathcal{L}} L\text{ is dense in }Y,}
and
 \eq{vd2}{\begin{aligned} \forall\text{ open  $W
 \subset                           Y$,  $\forall\,L\in \mathcal{L}$ such that }L  \cap W \neq
                          \varnothing, \text{ and  } \forall\,R\in \mathcal{R}
                          \\  
                           \exists\,L'\in \mathcal{L}
                          \text{ such  that }L' \cap L\cap W  \ne
                          \varnothing\text{ and }L' \not\subset R.\qquad \ 
\end{aligned}}
	We will say that  $\mathcal{L}$   is 
{\sl   totally
dense} if it is    totally
dense    relative to the collection of all closed nowhere dense subsets of $Y$. In this case  \equ{vd2} is equivalent to \eq{vd3}{\begin{aligned} \forall\text{ open  $W
 \subset                           Y$ and   $\forall\,L\in \mathcal{L}$ such that }L  \cap W \neq
                          \varnothing
                          \qquad \\  
                          \text{the closure of } \bigcup_{L'\in \mathcal{L}\,:\,L'\cap L \cap W  \ne
                          \varnothing}L'
                        \   \text{ has a non-empty interior}.
\end{aligned}}}
																																																																																																														
For example, for any $1\le d < n$ the collection {$\mathcal{L}$} of $d$-dimensional rational affine subspaces of $Y = \R^n$ is totally dense, with the union in  {\equ{vd3}} being dense in $\R^n$. More examples will be described 
{later in the paper}. On the other hand, the collection of straight lines in $\R^3$ that are parallel to one of the coordinate axes satisfies properties \equ{vd1}  {and  \equ{vd2} with $\mathcal{R} = $ all planes in $\R^3$, but not \equ{vd3}}.

{
{Here is another definition which will be important for us throughout the paper.} Let us say that $\mathcal{L}$ \textsl{respects} 
a subset $R$ of $Y$ 
if whenever $L\in \mathcal{L}$ 
is such that  $L\cap R$ has non-empty interior in $L$ (in the topology induced from $Y$), it follows that $L\subset R$. In other words, elements $L$ of $\mathcal{L}$ are not allowed to behave disrespectfully by intersecting $R$ in an open set and then wandering off. 
We will say that $\mathcal{L}$ 
\textsl{respects} $\mathcal{R}$ if it respects every $R\in \mathcal{R}$.  An example: if {each element of} $\mathcal{L}$ is either connected or has no isolated points, it respects any singleton $\{y\}\subset Y$. Another example is given by collections $\mathcal{L}$ and $\mathcal{R}$   consisting of analytic submanifolds of $\R^n$  (where we equip them with the metric inherited from the ambient space $\R^n$) such that every $L\in \mathcal{L}$ is connected, and every $R\in \mathcal{R}$ is closed; see Lemma \ref{lem: crucial} for more details.}

{Our final definition involves a collection $\mathcal{L}$  of proper closed subsets of 
$Y$, a Diophantine system $\mathcal{X} =
     \left( X, \mathcal{D}, \mathcal{H}\right)$, and a map $\varphi: Y\to X$. Let us say that \textsl{$\mathcal{L}$ is aligned with  $\mathcal{X}$ via  $\varphi$} if for any $L\in \mathcal{L}$ there exists $s\in\I$ such that  $d_{s}|_{\varphi(L)}\equiv 0$; in other words, if $\varphi(L)$ is contained in one of the resonant sets $d_s^{-1}(\{0\})$. In the special case when $Y = X$ and $\varphi = \Id$ we will simply say that $\mathcal{L}$ is aligned with  $\mathcal{X}$. Note that this property is independent both of the choice of  the heights functions $h_s$ and of a reparametrization of generalized
  distance functions as long as their zero sets are fixed.}
     \smallskip
     
     {Now we are ready to state the main theorem of the paper.
     \begin{theorem} \label{thm: special axiomatic theorem}
	Let $Y$ be a locally compact 
	metric space.
	For any $k\in \N$ suppose we are given a Diophantine system  \eq{xk}{\mathcal{X}_k = \big( X_k, \mathcal{D}_k = \{d_{k,s}: s\in\I\}, \mathcal{H}_k = \{h_{k,s}: s\in\I\}\big)}
        and
        a non-increasing 
     function $f_k: \R_{>0} \to \R_{>0}$, and let $\varphi_k$
        be a continuous map from $Y$ to $X_k$. 
	Let $\mathcal{L}$ and $\mathcal{R}$ be countable 
        collections of closed   subsets of $Y$ 
        such that 
        $\mathcal{L}$ is 
        totally dense relative to $\mathcal{R}$,
        respects $\mathcal{R}$, and
         is aligned with $\mathcal{X}_k$ via $\varphi_k$ for every $k\in\N$. 
	Then: 
	\smallskip 
	\begin{itemize}
	\item [{\rm (a)}] the set 
	\eq{specialset}{
         \bigcap_{k} \varphi_k^{-1}\big(\UA_{\mathcal{X}_k}(f_k)\big)
        \smallsetminus  \bigcup_{{R\in\mathcal{R}}} R
	}
	is 
	dense in $Y$. 
	\smallskip 
	\item [{\rm (b)}] In addition, if every $L\in\mathcal{L}$ is either connected or has no isolated points, the set \equ{specialset} is uncountable.
	\end{itemize}
	\end{theorem}}
\ignore{\begin{remark}\label{standard} Let $X = \R^n$, $\mathcal{A}$ be the set
  of rational affine hyperplanes, $H$ be the standard height, and
  $D(x,A) = H(A) \dist(x,A)$, where
  $$\dist(x,A) \df \inf_{a \in A} \|x-a\|$$
  is the Euclidean distance from $x$ to $A$. 
  Then $x\in \R^n$ is
  $(\mathcal{X},f)$-singular if and only 
  if  it is $f$-singular when viewed as a  linear form. But see below
  for more examples.

\end{remark}

\begin{remark}
The notion of $(\mathcal{X}, f)$-singularity belongs to the realm of
{\em uniform approximation}. If we only require
$\psi_{\mathcal{X}}(x,t) \leq f(t)$ for infinitely many $t$, then we
say that $x$ is {\em $f$-approximable}; this notion could be said to belong to the
realm of {\em asymptotic approximation}. Asymptotic approximation will
not concern us in this paper.
  \end{remark}
\smallskip

Below we will be working with {\em indexed collections} $\{L_i\}$. This is
shorthand for a function $i \mapsto L_i$; we stress that we do not
assume that this function is injective.

This is our main abstract theorem. }

        \begin{remark}\label{doesnotmatter} \rm 
{A few comments are in order.}
\begin{itemize}
\item[(1)] 
In the above statement we are working with an {\em indexed collection} $\mathcal{X}_k$, which is
shorthand for a function  $k \mapsto \mathcal{X}_k$;
we do not 
assume that this function is injective. Also it might happen  that the
underlying set $X_k$ of $\mathcal{X}_k = (X_k, \mathcal{D}_k,
\mathcal{H}_k)$ arises several times, with distinct generalized
distance functions or height functions.

\item[(2)]  Neither the approximation functions $f_k$  nor  the heights
$\mathcal{H}_k$ 
appear in the
conditions of the theorem. {Moreover,  the collections $\mathcal{D}_k$ of generalized distance functions appear only through their zero sets  $d_{k,s}^{-1}(\{0\})$.} Thus if the 
{assumptions of the above theorem} are satisfied,
then they are satisfied for all positive {non-increasing }functions
$f_k$, all choices of the heights, {and all choices of generalized
  distance functions as long as their zero sets are fixed.}
\item[(3)] {It is easy to see that} Theorem \ref{thm: special axiomatic theorem} 
holds
for finite intersections, 
that is, when the {set} of
indices $k$ is finite; this follows immediately from the statement of
the theorem, taking $f_k, \varphi_k, \mathcal{X}_k$ the same for all
sufficiently large $k$.
\end{itemize} 
\end{remark}
          Theorem \ref{thm: special axiomatic theorem} 
 {is} proved
          in \S \ref{abstr},
    {and in \S\ref{dioph} we 
show}  that it implies a
          strengthening of Theorem \ref{thm: classical theorem}(a),
          {making it} possible to avoid countably many proper analytic
          submanifolds in $M_{m,n}(\R)$ (see Theorem \ref{cor:
            matrices}).  
        In the remainder of the introduction 
we  present a list of additional
        applications of Theorem \ref{thm: special axiomatic theorem}.

\subsection{Matrices 
{uniformly approximable} with 
  different weights}\label{subsec: weights intro}
Many results in
   \da\ extend to  {\sl approximation with  weights}, an approach in
   which one 
   treats differently  {different linear forms $A_i $ (the rows
   of $A$), as well as different} components of $\vq$.
   
For  an
$(m+n)$-tuple of positive weights\footnote{{In the literature it is often assumed that the weight vectors are normalized so that $\sum_i\alpha_i = \sum_j\beta_j = 1$,  but for {this work} it makes no difference, since our results are valid for arbitrary choice of approximating functions.} 
}  
{$    \vw = (\va,\vb) \in \R_{>0}^{m+n} $, where  $\va  = (\alpha_1,\dots,\alpha_m)\in\R_{>0}^m$ and $\vb = (\beta_1,\dots,\beta_n)\in\R_{>0}^n$,} 
let us introduce \textit{quasi-norms}
\begin{align*}
\|\x\|_{\va}\df\max_i|x_i|^{1/\alpha_i} \quad\textrm{and}\quad \|\y\|_{\vb}\df\max_j|y_j|^{1/\beta_j}
\end{align*}
{on $\R^m$ and $\R^n$ respectively}.
Then, for 
$f$
as above, one  says that  \amr\ is $(f, \vw)$-{\sl uniform},
denoted by $A\in \UA_{m,n}(f, \vw)$, if 
for every sufficiently
large $t$  one can find
$\vq  \in \Z^n\nz$ and $\vp \in \Z^m$
 with
\eq{diweighted}{
\|A\vq -\vp\|_{\va}\le f(t)
  \ \ \ \mathrm{and}  
\ \ \|\vq\|_{\vb}\le t
.}
 In other words, we are considering the solvability of the
 system \begin{equation*}\label{weightsystem} 
  \begin{cases} \  \ \left| A_i   \vq  - p_i\right| \le  f(t)^{\alpha_i}, &  i = 1,\dots,m;\\ \qquad \qquad  \ \ \,|q_j|\le t^{\beta_j},& j = 1,\dots,n.\end{cases}
\end{equation*}
See \cite{KR} for a discussion. Clea{rly the unweighted case
corresponds to the choice {$\vw 
= (1,\dots,1)$}. {As in Definition \ref{Dir}, let us denote by  $\UA^*_{m,n}(f, \vw)$ the set of totally
irrational $A\in \UA_{m,n}(f, \vw)$.}

{Clearly for any $    \vw = (\va,\vb)$ and $f$ as above 
one has  \eq{reduction}{\UA_{m,n}(f, \vw)\supset \UA_{m,n}\big(\tilde f\big),\text{ where }\tilde f(t) \df f\big(t^{1/\min_j\beta_j}\big)^{\max_i\alpha_i}.}
Hence it immediately follows from Theorem \ref{thm: classical theorem}(a) that   $\UA^*_{m,n}(f, \vw)$	is
     uncountable and dense.}
Using {Theorem \ref{thm: special axiomatic theorem}} {one} can extend {this} to  {matrices} which are simultaneously $(f_k, \vw_k)$-uniform for countably many $k$.   

\begin{theorem}\label{thm: intersection with weights}
  {Let $n \geq 2$ and  $m \in \N$. 
  \begin{itemize}
  \item[\rm (a)] For any sequence $\vw_1, \vw_2, \ldots \in
  \R_{>0}^{m+n}$ of  weight vectors and any sequence  $f_1, f_2, \ldots$ of positive non-increasing
  functions, the intersection   $\bigcap_k\UA^*_{m,n}(f_k,
  \vw_k)$ is {dense and} uncountable. 
   \item[\rm (b)]
  {The intersection $\bigcap_{\vw \in \R_{>0}^{m+n}}\UA^*_{m,n}(f,
  \vw)$ is {dense and} uncountable for any  positive \linebreak non-increasing $f$}.
  \end{itemize}
 \ignore{ there are uncountably many totally irrational $A \in M_{m,n}(\R)$
  such that $A \in \UA_{m,n}(f,
  \vw)$ for any weight $\vw \in \R_{>0}^{m+n}$. Moreover, if  
 $\vw_1, \vw_2, \ldots \in
  \R_{>0}^{m+n}$ are  weight vectors, and 
  $f_1, f_2, \ldots$ {are} positive non-increasing
  functions, then there are uncountably many totally irrational
  $A\in M_{m,n}(\R)$ that are $(f_k, \vw_k)$-uniform,
  for all $k \in \N$.}}
\end{theorem}

In \S\ref{dioph} we 
 prove {an even} stronger statement, 
{with extra uniformity in choosing approximating vectors} and making it possible to avoid
countably many proper 
analytic submanifolds (see Theorem \ref{cor: weighted
  matrices}).

\subsection{{Uniform approximation of higher order on submanifolds}}\label{mflds}
In \S\ref{sbmfolds} we 
show that in the row vector case ($m=1$) vectors satisfying the
conclusions of Theorem \ref{thm: intersection with
  weights}  exist on many manifolds and fractals. 
For example the following theorem, which  is a special case of Theorem \ref{cor: linear forms manifolds}, was already proved in \cite{KMW}: 
\begin{theorem}\label{cor: linear forms manifolds simple}
Let $Y$ be a connected analytic submanifold  of $\R^n$  of   dimension at least  $2$ not contained in a rational affine hyperplane, and let $f$ be  any
positive non-increasing function. Then  the intersection of  $\UA^*_{1,n}(f)$ with $Y$ is dense and uncountable. 
\end{theorem}
Our proof, based on Theorem \ref{thm: special axiomatic theorem}, makes it possible to streamline the argument and derive other related results. In particular, one benefit of our level of abstraction is that we can similarly study
uniform approximation \textit{of higher order}. 
Let us introduce the following 

\begin{definition}\label{orderg} \rm
Given a  
function $f: \R_{>0} \to \R_{>0}$ and $1\le g \le n$, say 
that $\mathbf{x} \in \R^n$ is {\sl $f$-uniform of order $g$} (notation:   $\x\in\UA_{1,n}(f; g)$, or $\x\in\UA^*_{1,n}(f; g)$ if $\x$ is totally irrational) 
if 
for every sufficiently 
large $t>0$  one can find 
linearly independent $\mathbf{q}_1, \ldots,
  \mathbf{q}_g$ in $\Z^n$ 
  and $p_1, \ldots, p_g \in \Z$ such that
  $$\text{ for all } i \in \{1, \ldots, g\}, \ \ \  \|\mathbf{q}_i\|
  \leq t \ \text{ and } \ | \mathbf{q}_i \cdot \mathbf{x} -p_i| \leq
  f(t).
  $$
\end{definition}
\ignore{  If $\va  = (\alpha_1,\dots,\alpha_m)\in\R_{>0}^m$, one can similarly define $(f,\va)$-uniform approximability of order $g$ by replacing $\|\mathbf{q}_i\|$ in the above system of inequalities with $\|\mathbf{q}_i\|_\va$. Denote by $\UA_{1,n}(f, \va; g)$ the set of those, and by $\UA^*_{1,n}(f, \va; g)$ the set of totally irrational vectors in $\UA_{1,n}(f, \va; g)$.}
  Note that 
  any rational affine subspace of
codimension $g$ is an intersection of $g$ rational affine hyperplanes with linearly independent normal vectors,
and hence its every point is $f$-uniform of order $g$ for any
 $f$. The following generalization of Theorem \ref{cor: linear forms manifolds simple} is a special case of Theorem \ref{cor: linear forms manifolds}.
\begin{theorem}\label{thm: using transference}
Let $Y$ be a connected analytic submanifold  of $\R^n$  of   dimension   $d\ge 2$ not contained in a rational affine hyperplane,  let $f$ be  any
positive non-increasing function, {and let $1\le g \le d-1$}.
Then  the intersection of $Y$ with $\UA^*_{1,n}(f; g)$ is dense and uncountable.
\end{theorem}

\subsection{Vectors which are $k$-singular for all $k$}\label{ksing}
A vector {$\mathbf{x} 
 \in \R^n$} is called {\sl  $k$-singular} {\cite{KPW}} if for any
$\vre>0$ there is $Q_0>0$ so that for all $Q > Q_0$ there exists a
polynomial {$P \in \Z[X_1, \ldots, X_n]$} such that $$\deg (P) \leq k, \
H(P) \leq Q,\text{ and }{|P(\mathbf{x})| \leq \vre Q^{-N(k,n)}}.$$ Here $H(P)$ is the {\sl  height} of $P$, defined as the maximum of the
absolute value of the 
coefficients of $P$, $\deg(P)$ is the total degree of $P$, and 
\begin{equation}\label{eq: def n}
{N(k,n) = \left( \begin{matrix} k+n \\ n\end{matrix} \right) -1}.
\end{equation}
{Note that when $k=1$, $k$-singularity is the same as singularity of
$\mathbf{x}$ considered as a row vector in  $M_{1,n}(\R)$, and
the existence of singular vectors in certain manifolds and fractals is
one of the main results of \cite{KMW}.}

The study of $k$-singular
vectors is related to approximation of $\mathbf{x}$ by vectors with
algebraic coefficients of degree at most $k$; see
e.g.\;\cite{Yu}. Following \cite{KPW}, let us say that $\mathbf{x}$ is {\sl 
$k$-algebraic} if there is a nonzero polynomial ${P \in \Z[X_1, \ldots,
X_n]}$ such that $P(\mathbf{x})=0$ and $\deg P =k$, and we say it is
{\sl  algebraic} if it is $k$-algebraic for some $k$. It is easy to see
that a $k$-algebraic  vector is $k'$-singular for every $k'\geq k$,
and it was shown in \cite{KPW} that for {$n \geq 2$} and for each $k$ there
are $k$-singular vectors which are not $k$-algebraic. In 
{\S \ref{sec: k singular}}
we show:
\begin{theorem}\label{thm: k singular all k}
If {$n \geq 2$}, then there {is a dense  uncountable set of} {$\mathbf{x} \in \R^n$}
such that $\mathbf{x}$ is $k$-singular for all $k \in \N$, but
$\mathbf{x}$  is not algebraic. 
\end{theorem}
Theorem \ref{thm: k singular all k} answers a question raised in
\cite[\S 2.3]{KPW} regarding the existence of vectors which are $(k,
\vre)$-Dirichlet improvable for some fixed $\vre$, for all positive
integers $k$. 
{Note that 
we are going to prove a  statement much more general than  Theorem \ref{thm: k singular all k}; namely, we will discuss vectors that are  uniformly $f_k$-approximable  of   degree $k$ for each $k$, where $f_k$ is a sequence of arbitrary positive non-increasing functions; moreover, those vectors will be found in many manifolds and
fractals.}

\subsection{
{$\Sing^*_{1,n}+ \Sing^*_{1,n} = \R^n$ and generalizations}}
{\ignore{\st{For a given $d$ and a positive function $f$,  let $\UA_d(f) \df
\UA_{1,d}(f)$ (and $\UA^*_d(f) \df
\UA^*_{1,d}(f)$) 
denote the set of (totally irrational) 
$f$-uniformly approximable vectors in
$\R^d$. If $\vw$ is a collection of weights on $\R^d$, we denote by
$\UA_d(f, \vw)$ and $\UA^*_d(f, \vw)$ the set of
(totally irrational) $(f, \vw)$-uniformly approximable vectors in
$\R^d$.} I don't think we need this new notation, it's already confusing when we switch from rows to columns, and there is nothing too complicated with the original notation. I also removed weights from the introduction to make it easier to handle.}
{Here we} consider the case $m = 1$. It is {known \cite{dfsu} that for rapidly decaying $f$ 
the sets  $\UA_{1,n}(f)$ 
have very small \hd. The next theorem shows that when $n\ge 3$ the sum of two such sets is the whole space.}
\begin{theorem}\label{thm: schleischitz extension}
  If $n \geq 3$, then for any {two  non-increasing functions $f_1,f_2:\R_{>0} \to \R_{>0}$} 
 one has
 \eq{sumset}{{\UA^*_{1,n}(f_1) + \UA^*_{1,n}(f_2)} = \R^n. }
\end{theorem}
See Schleischitz \cite[\S
3]{Schleischitz} for related work. 
{We 
prove this theorem in \S\ref{sumsets} 
and discuss several extensions.}

\subsection{Improved rates of singularity  {for column vectors}} \label{subsec: rates intro} 
Recall from Theorem \ref{thm: classical theorem} that for $n>1$, any $m\in\N$ and  
any positive non-increasing $f$ one
can find {an $m\times n$} matrix which is totally irrational and 
$f$-uniform. For {$n=1$ and $ m>1$}, that is for the case of column
vectors, this is no longer the case. {Indeed, as was mentioned in {Remark \ref{exponents}}, 
 for all totally irrational  $\x\in\R^m \cong M_{m,1}(\R)$ one has $\hat
\omega(\mathbf{x})\le 1$, and, according to Theorem \ref{thm: classical
  theorem}(b), this  maximum value is attained on an uncountable dense set.} 

It is natural to inquire about the value of $\hat \omega$ that 
{vectors} in
certain fractals and submanifolds may attain. {One can approach this question by the standard transference argument as in  \cite[Ch.~V, Thm.\ II]{Cassels_book}  and \cite{Jarnik38} and show that any column vector $\mathbf{x}\in {\R^n\cong M_{n,1}(\R)}$ 
with $\hat
\omega(\mathbf{x}^T) = \infty$ satisfies \eq{transfbound}{\hat
\omega(\mathbf{x}) \ge \frac1{{n}-1}.} Therefore an application of Theorem \ref{cor: linear forms manifolds simple} coupled with transference will produce, for manifolds $Y$ as in that theorem,
a dense uncountable set of $\x\in Y$ satisfying  \equ{transfbound}.}

{It is natural to seek an improvement of the above bound; however replacing $ \frac1{{{n}}-1}$ with $1$ on an arbitrary analytic submanifold of $\R^{{n}}$ is an impossible task. Indeed,} in \cite{KM} it was shown
that certain submanifolds do not contain {column vectors} $\mathbf{x}$ with $\hat
\omega(\mathbf{x}) = 1$. Namely, letting {$H_{{n}} \in
\left(\frac12,1\right)$ be the unique 
positive root of the equation $x + \cdots + x^{{{n}}+1} =1$, it was proved that any $\mathbf{x}$ in the sphere $\{(x_1, \ldots,
x_{{n}}): \sum x^2_i =1\} $ satisfies $\hat \omega(\mathbf{x})  \leq
H_{{{n}}-1}$, and any $\mathbf{x}$ in the paraboloid $\{(x_1, \ldots, x_{{n}})
: x_{{n}} = \sum_{i<{{n}}} x_i^2\}$ satisfies $\hat \omega(\mathbf{x})  \leq
H_{{{n}}}$. 
Furthermore, in \cite{KMW} it was shown that  there are
$d$-dimensional affine subspaces $L \subset \R^{{n}}$ such that
any $\mathbf{x} \in L$ satisfies $\hat \omega(\mathbf{x})
\leq \frac{d+1}{{{n}}-d}, $ which is strictly smaller than {$1$ for $d < \frac{{n}-1}2 $}.
In \S\ref{sec: transference}, using uniform approximation of higher order, i.e.\ Theorem \ref{thm: using transference}}, {together with a transference argument}, we are able to improve 
{the bound in  \equ{transfbound}}
and prove the following

\begin{theorem}\label{thm: improvement rate}
{Let $Y$ be a {connected analytic submanifold  of ${\R^n\cong M_{n,1}(\R)}$  of   dimension 
 $d\ge 2$ not contained in a rational affine hyperplane. 
Suppose that  a {continuous non-decreasing}  function $f$ satisfies
 \begin{equation}\label{eq: goes to infty monotonically }
t^{\frac{1}{{n}-d+1}}\cdot 
f(t) \to \infty\,\,\,\, \text{ monotonically as }  t\to \infty.
\end{equation}
Then  the intersection of  $\UA^*_{{n},1}(f)$ with $Y$ is dense and uncountable. 
In particular, there exists a dense and  uncountable set of totally irrational $\mathbf{x} \in
Y$ satisfying
$\hat\omega(\mathbf{x}) \geq \frac{1}{{n}-d+1}$}. }
\end{theorem}


    \begin{remark}\label{dt} \rm 
{
Using methods from \cite{Barak_GAFA} and \cite{KMW} and a dynamical interpretation of uniform approximation through divergence of trajectories in the space of lattices, Datta and Tamam    \cite{DT} recently proved a weighted version of Theorem \ref{thm: improvement rate} independently of the present paper and not using transference. More precisely, \begin{itemize} \item    \cite[Theorem 1.1]{DT} produces totally irrational vectors on analytic submanifolds that are singular with respect to multiple weights in the spirit of Theorem \ref{thm: intersection with weights};
\item   \cite[Theorem 1.2]{DT} constructs vectors in $\UA^*_{n,1}(f, \vw)\cap Y$ for any weight $\vw$ and any affine subspace $Y$ of $\R^n$ not contained in a rational affine hyperplane, with optimal decay conditions on approximating functions $f$; this in particular proves Theorem \ref{thm: improvement rate} for $Y$ being an affine subspace of $\R^n$. 
\end{itemize} 
It is likely that the above results can  be derived from the main theorem of the present paper by combining Theorem \ref{thm: using transference} with a weighted transference argument as in \cite{many} and \cite{G}.}\end{remark}

\ignore{See Remark 5.2 of \cite{KMW}. 

The proof of Theorem \ref{thm: improvement rate} relies on a new
result concerning higher order uniform approximations. 

 Recall also that an affine subspace $L \subset
\R^d$ is called {\em totally irrational } if it is not contained in a
proper rational affine subspace of $\R^d$. If $L = \{\mathbf{x}\}$ is
a singleton (affine subspace of dimension 0) then being totally
irrational is equivalent to saying that the coordinates $x_i$ of
$\mathbf{x}$, together with 1, are linearly independent over $\Q$.  }
{

	\subsection{Other applications} \label{subsec: other} {In the last section of the paper we briefly survey a few other settings where our main theorem can be applied. This includes: a relationship between the conclusion of our main theorem and the absence of  Kan--Moshchevitin phenomenon as exhibited in \cite{KaMo}; \da\ with restrictions on $\vp$ and $\vq$; inhomogeneous approximation; approximating subspaces of $\R^d$ by rational subspaces.  Proofs will appear in a sequel to this paper.}

\bigskip

\noindent {\bf Acknowledgements.}
The authors thank Yitwah Cheung, Vasiliy Neckrasov, Nicolas de Saxc\'e and Damien Roy for helpful discussions.  
{Thanks are also due to the anonymous referee for a careful reading of the paper and many insightful comments.}

\section{Proof of the abstract result
}\label{abstr}

\begin{proof}[Proof of Theorem \ref{thm: special axiomatic theorem}]
{Recall that the statement of the theorem 
involves two countable collections $\mathcal{L}$ and $\mathcal{R}$ of subsets of $Y$. In the course of the proof we will \textit{index} those collections by $\N$, writing  
$\mathcal{L} = \{L_i: i\in\N\}$ and $\mathcal{R} = \{R_\ell: \ell\in\N\}$. That is, essentially we will work with functions $i \mapsto L_i$ and $\ell \mapsto R_\ell$,
and will not 
assume  these functions to be injective.} 
 \ignore{We are given two countable collections $$\mathcal{L} = \{L_i: i\in\N\}\text{ and }\mathcal{R} = \{R_\ell: \ell\in\N\}$$ of closed   subsets of $Y$ satisfying the assumptions of the theorem.}

{Take a non-empty open   $W\subset Y$, 
which, since  $Y$ is 
        locally compact, we can assume to be relatively compact.} 
       We will inductively construct a nested sequence of
        open sets $U_\ell \subseteq U_0 \df  {W}$, an increasing sequence 
	$T_\ell \to \infty$, and a sequence of indices $i_\ell$ so that for all $\ell \in \N$:
	\begin{enumerate}[label=(\roman*)]
		\item 
		$\varnothing\neq \ov{ U_{\ell}}\subseteq U_{\ell-1}$,
		\item $ L_{i_\ell}  \cap  U_\ell \ne \varnothing$ 
		and ${U_\ell \cap 
{R_\ell} 
		=\varnothing}$, 
		\item for all $1\le k \le \ell$ there exists 
                   {$s = s(k,\ell)\in \I$
                  such that $d_{k,s(k,\ell)}\equiv 0$ on 
                 $\varphi_{k}(L_{i_\ell})$
                  and $h_{k,s(k,\ell)} \le
                  T_\ell$,}
                  
		\item[(iv)] for all 
		{$1 \le k \le \ell-1$} and all $x \in U_{\ell},$
		$d_{k,s(k,\ell-1)}\big(\varphi_k(x)\big) <
                f_k(T_{\ell})$, where $s(k,\ell-1)$ is as
                  in (iii) (for $\ell-1$). 
	\end{enumerate}

	
	\medskip

        \noindent\textbf{Step 1: Sufficiency.} 
	Let us first check that this is sufficient {for Part (a) of the theorem}. 
	First observe that (i) and the relative compactness of
          ${W}$ imply that $ \bigcap\limits_{\ell \in \N} \ov{ U_\ell} $
          is non-empty.  
	We claim that $$x \in \bigcap\limits_{\ell \in \N} \ov{ U_\ell} \implies 
 \forall\,k\in\N, \	\text{ $\varphi_k(x)$ is $(\mathcal{X}_k,f_k)$-{uniform}. }
	$$ 
	Indeed, for any $k$ take $T\ge T_{k}$, and let $\ell$ be such
        that $T_\ell\le T < T_{\ell+1}$; then clearly $\ell \ge k$. Take 
        {$s = s(k,\ell)$}. Then 
        $h_{k,s} \le
                   T_\ell\le T $ by (iii), and by (iv)
        and since $f_k$ is non-increasing and $x \in U_{\ell+1}$, we have that
        $$
        {d_{k,s}\big(\varphi_k(x)\big) 
        < f_k({T_{\ell+1}})\le
        f_k(T).}$$ Therefore
        $\varphi_k(x)$ is $(\mathcal{X}_k,f_k)$-{uniform}.  
	Also from (ii) it follows that $x\notin \bigcup_{\ell} R_\ell$.
       Thus $x$ {belongs to  the set  
        \equ{specialset},}  {which implies its density and finishes the proof of (a)}.
	
	

	\medskip

	\noindent\textbf{Step 2: Base case of induction.}
	By 
        {the total density of $\mathcal{L}$ relative to $\mathcal{R}$} there exists
        $i_1 $ 
        so that $$L_{i_1} \cap {W} \ne \varnothing \
        \ \text{ and } \ \ L_{i_1}\not\subset R_1.$$	
{Since $\mathcal{L}$ is aligned with $\mathcal{X}_1$ via $\varphi_1$}, there 
{exists $s = s(1,1)$}
such that
$\varphi_1(L_{i_1}) $
{is contained in} ${d_{1,s}^{-1}(\{0\})}$. Choose $T_1$ so that $T_1 > 
{h_{1,s}}$, and let $z \in
L_{i_1} \cap {W}$.
Choose a neighborhood $\hat{U}_1$  of $z$ so that $\overline{\hat{U}_1} \subset
{W}$. 
{Since $\mathcal{L}$ respects $R_1$ and   
the latter} is closed, there is an open set $U_1 \subset 
\hat{U}_1$ such that $L_{i_1}  \cap U_1 \neq \varnothing$ and
$U_1 \cap 
{R_1} = \varnothing.$ This choice ensures that 
(ii) holds; (iii) holds by the choice of 
${s}$ and $T_1$, and (i)
holds since $U_1 \subset \hat{U}_1$ and $U_0 ={W}$.  Item (iv) is
vacuous for $\ell=1$.        This completes the base case. 

\medskip
	 
	\noindent\textbf{Step 3: Inductive step.}
	Assume {that} we have $U_\ell, T_\ell, i_\ell$ satisfying the inductive
        hypotheses. {In view of  
        \equ{vd2} and (ii)}, there exists
        $i_{\ell+1}$ so that \[L_{i_{\ell+1}}\cap L_{i_\ell}
          \cap U_\ell \ne\varnothing \text{ and }L_{i_{\ell+1}}
          \not\subset R_{\ell+1}.\]
       {Since for any $k\in\N$ the collection $\mathcal{L}$ is aligned with $\mathcal{X}_k$ via $\varphi_k$}, for each $k = 1, \ldots, \ell+1$ there is 
         {$s(k,\ell+1)\in \I$}
         such that $\varphi_k(L_{i_{\ell+1}})
        \subset 
        {d_{k,s(k,\ell+1)}^{-1}(\{0\})}$. Let
        $$T_{\ell+1} > \max \left(T_\ell \cup \{
        {h_{k,s(k,\ell+1)}} : k=1,
          \ldots, \ell+1\}  
          \right),$$
and let $z \in
L_{i_{\ell+1}} \cap L_{i_\ell} \cap U_\ell $. Since for
$k = 1, \ldots, \ell$ the functions $
{d_{k,s(k,\ell)}} \circ \varphi_k
$ are continuous and vanish
at $z \in L_{i_\ell}$, there is a neighborhood $\hat{U}_{\ell+1}$ of $z$
such that $\overline{\hat{U}_{\ell+1} }\subset
U_\ell$ and $$
{d_{k,s(k,\ell)}}\big(\varphi_k(y)\big) < f_k(T_{\ell+1})\text{  for all }y \in
\hat{U}_{\ell+1}.$$ 
{Again, since  $R_{\ell+1}$ is closed and is respected by $\mathcal{L}$}, 
it follows that there is an open set $U_{\ell+1} \subset 
\hat{U}_{\ell+1}$ such that $$L_{i_{\ell+1}} \cap U_{\ell+1} \neq \varnothing\text{ and
}U_{\ell+1} \cap   R_{\ell+1} = \varnothing.$$
This choice, and the inductive hypothesis, ensure that 
(ii) holds for $\ell+1$; (iii) holds for $\ell+1$ by the choice of 
{$s(k,\ell+1)$} and $T_{\ell+1}$, and (i)
and (iv)
hold for $\ell+1$ since $U_{\ell+1} \subset \hat{U}_{\ell+1}$.   {This finishes the proof of (a).}  \\

\noindent\textbf{Step 4: Part (b).} {
The second part of Theorem \ref{thm: special axiomatic theorem} easily follows from (a). Indeed, arguing by contradiction, assume that the set \equ{specialset} consists of countably many points $\mathcal{R}_0\df \{y_j: j\in\N\}$. Then one can replace $\mathcal{R}$ with $\mathcal{R}\cup \mathcal{R}_0$ which will still satisfy \equ{vd2} and  be respected by $\mathcal{L}$ (here we use the fact that   every $L\in\mathcal{L}$ is either connected or has no isolated points).  Thus the set   \equ{specialset} with this new choice of $\mathcal{R}$ will be empty, contradicting Theorem~\ref{thm: special axiomatic theorem}(a).}
\end{proof}

\section{Simultaneous approximation}\label{dioph}

{In this section} we 
{consider} the {standard} Diophantine  system $\mathcal{X}_{m,n}$ as in  \equ{classical}.
{We are going to apply Theorem~\ref{thm: special axiomatic theorem} with $\varphi_k = \Id$ for each $k$ and with the collection $\mathcal{L}$ 
  parametrized by the direct product of $\Z^m$ and $\Z^n\nz$.} Namely we will {consider the collection \eq{ell}{\begin{aligned} \mathcal{L} \df \{L_{\vp,\vq} : \vq \in\Z^n\nz, \ \vp\in\Z^m\},\\
 \text{  
where }L_{\vp,\vq} \df \{A\in\mr: A\vq = \vp&\}.
\end{aligned}}
Let us prove the following}

       \begin{proposition} \label{verydensematrices}
	Let $m,n\in\N$ with $n > 1$, {and let $Y$ be a non-empty open subset of $\mr$. Then the collection \eq{intwithY}{\{L_{\vp,\vq} \cap Y: \vq \in\Z^n\nz, \ \vp\in\Z^m\}}	is 
	totally dense.}
	 \end{proposition}


\begin{proof} Clearly $M_{m\times n}(\Q)$ is dense in $\mr$, and,
  furthermore, \linebreak $L_{\vp,\vq}\cap M_{m\times n}(\Q)$ is dense
  in $L_{\vp,\vq}$ for any  $\vq \in\Z^n\nz$ and
  $\vp\in\Z^m$. Moreover, any $B\in M_{m\times n}(\Q)$ lies in an
  $L_{\vp,\vq} $ for some $\vq \in\Z^n\nz$ and $\vp\in\Z^m$. This
  implies \equ{vd1}. 

Now fix $(\vp_0,\vq_0) \in\Z^m\times(\Z^n\nz)$ {such that $L_{\vp_0,\vq_0}\cap Y\ne\varnothing$}, and  choose an  open $W
 \subset                           {Y}$ which intersects {$L_{\vp_0,\vq_0}$} non-trivially. Then {one can pick} an arbitrary
 $$B\in M_{m\times n}(\Q)\cap L_{\vp_0,\vq_0}\cap W,$$  let $E_B$ be the union of all the subspaces  $L_{\vp,\vq}$ containing $B$ (this clearly includes $L_{\vp_0,\vq_0}$),
{and} write
 $$
 \begin{aligned}
 E_B &=\quad \bigcup_{(\vp,\vq) \in\Z^m\times(\Z^n\nz): B\vq = \vp}\{A\in\mr: A\vq = \vp\}\\
 &=\quad \ \bigcup_{\vq \in\Z^n\nz, \, B\vq\in\Z^m}\{A\in\mr: A\vq  =
 B\vq\} \\ &= \quad B +  \bigcup_{\vq \in\Z^n\nz, \,
   B\vq\in\Z^m}\{C\in\mr: C\vq = 0\}. 
 \end{aligned}
 $$
Note that the set of $\vq \in\Z^n$ such that  $B\vq\in\Z^m$ contains
$N\Z^n$ for some $N\in\N$. Therefore one has  
$${E_B} \supset B + \{C\in M_{m\times n}(\Q): \rank C < n\}.$$  
If $n>m$ it is clear that 
$E_B$ contains $M_{m\times n}(\Q)$, which readily implies the {total density of the collection  \equ{intwithY}} in an even stronger form: namely, that the union in \equ{vd3} is dense in $\mr$.


 In general we see that
 $\overline{E_B} = B + R_{<n}$, where 
 $$R_{<n} \df \{C\in\mr: \rank C < n\},$$
which is a proper algebraic subvariety of $\mr$ if $n \le m$. Now let us consider the union of the sets  $E_B$ over all $B\in M_{m\times n}(\Q)\cap L_{\vp_0,\vq_0}\cap W$, and then take the closure, which is easily seen to have the following form:
$$
\overline{\bigcup_{B\in M_{m\times n}(\Q)\cap L_{\vp_0,\vq_0}\cap W} E_B} = \{B\in \overline{W}: B\vq_0 = \vp_0\} + R_{<n}.
$$
If $n=1$ then $E_B = \{B\}$, and the above closure coincides with $ {L_{\vp_0,\vq_0}}\cap \overline{W}$, hence no total density. Now suppose that $1 < n \le m$, and let $D\in\GL_n(\R)$ be such that $\vq_0 = D\e_n$. Then one can write
 $$
 \begin{aligned}
 \{B\in \overline{W}: B\vq_0 = \vp_0\} + R_{<n} &=  \{B\in \overline{W}: BD\e_n = \vp_0\} + R_{<n} \\ &= \big(\{A\in \overline{W}{D}: A\e_n = \vp_0\} + R_{<n} \big)D^{-1},
 \end{aligned}$$
 since $R_{<n}$ is invariant under right-multiplication by invertible matrices. Thus it suffices to prove {the claim under the assumption that $\vq = \e_n$, that is, to show that} for any non-empty open $W\subset\mr$ and any $\vp_0\in\Z^m$, the set  {$\{A\in  \overline{W}: A\e_n = \vp_0\} + R_{<n} $ 
contains 
some open neighborhood $V$ of   $L_{\vp_0,\e_n}  \cap W$.} 
 
 {For that} it will be convenient to write elements of $\mr$ in a column-vector notation. 
 {Namely, we can write
 $$L_{\vp_0,\e_n}\cap W = \big\{[\begin{matrix} \a_1 & \cdots & \a_{n-1}& \vp_0\end{matrix}]: [\begin{matrix} \a_1 & \cdots & \a_{n-1}\end{matrix}]\in W'\big\}$$
for some non-empty open subset $W'$ of $M_{m\times(n-1)}$. Then it becomes clear that we can   take 
$$V \df \big\{[\begin{matrix} \a_1 & \cdots & \a_{n-1}& \b\end{matrix}]: [\begin{matrix} \a_1 & \cdots & \a_{n-1}\end{matrix}]\in W',\ \b\in\R^m\big\},$$
simply because any matrix $[\begin{matrix} \a_1 & \cdots & \a_{n-1}& \b\end{matrix}]$ with $[\begin{matrix} \a_1 & \cdots & \a_{n-1}\end{matrix}]\in W'$ can be written as $$ [\begin{matrix} \a_1 & \cdots & \a_{n-1}& \vp_0\end{matrix}] + [\begin{matrix} {\bf 0} & \cdots & {\bf 0}& \b - \vp_0\end{matrix}]\in (L_{\vp_0,\e_n}\cap W) + R_{<n}.$$
\ignore{That is, we are looking for all \linebreak matrices $B = [\begin{matrix} \b_1 & \cdots & \b_{n-1}& \b_n\end{matrix}]\in\mr$ which can be written as $A + C$, where 
{$$A \in L_{\vp_0,\e_n}\cap W = \big\{[\begin{matrix} \a_1 & \cdots & \a_{n-1}& \vp_0\end{matrix}]: [\begin{matrix} \a_1 & \cdots & \a_{n-1}\end{matrix}]\in W'\big\}$$
for some non-empty open subset $W'$ of $M_{m\times(n-1)}$,  and $$C \in \big\{ [\begin{matrix} \c_1 & \cdots & \c_{n-1}& \c_n\end{matrix}]: 
 \c_1, \dots, \c_n\text{ are linearly dependent}\big\}.$$
Now let us observe} that the last column of the matrix $C$ is
uniquely determined by that of $B$, that is, one can put $\c_n\df \b_n - \vp_0$. 
Thus it is clear that $B$  can be written as $A + C$ as above if and only if $\b_n$ is arbitrary, and $[\begin{matrix} \b_1 & \cdots & \b_{n-1}\end{matrix}]$ belongs to $$W' + \big\{ [\begin{matrix} \c_1 & \cdots & \c_{n-1}\end{matrix}] : [\begin{matrix} \c_1 & \cdots & \c_{n-1}&\b_n - \vp_0\end{matrix}]\in R_{<n}\big\},$$
which is a non-empty open subset of $M_{m\times(n-1)}$ {containing $W'$.}} This finishes the proof.}
 \end{proof}

{We can now use the above proposition to furnish the}
\begin{proof} [Proof of  Theorem \ref{thm: classical 
   theorem}(a)]
   Indeed, it is clear that the collection $\mathcal{L}$ as in  \equ{ell} is aligned with $\mathcal{X}_{m,n}$, and its total density is provided by Proposition \ref{verydensematrices}. 
   Let $\mathcal{R}$ be the collection of all proper rational affine subspaces of $\mr$. Obviously if $L$ and $R$ are two subspaces of a finite-dimensional vector space such that  $L\cap R$ has non-empty interior in $L$, then $\dim(L\cap R) = \dim(L)$, which implies that $L\subset R$. Hence $\mathcal{L}$ respects any subspace of $\mr$. This verifies all the conditions of 
   Theorem \ref{thm: special axiomatic theorem}; and since elements of $\mathcal{L}$ 
{have no isolated points}, part (ii) implies the uncountability of $\UA^*_{m,n}(f)$ for any positive non-increasing $f$.\end{proof}
   
{More generally,  one can strengthen Theorem \ref{thm: classical 
   theorem}(a) using the fact that $\mathcal{L}$ respects  {closed} analytic submanifolds of $\mr$. Recall that ${\Psi}: U\to \R^k$, where $U$ is an open subset of $\R^d$ with $d \leq k$, is called a {\sl real analytic immersion} if it is injective, each of its
coordinate functions ${\Psi}_i :
U\to \R \ (i=1, \ldots, k)$ is infinitely differentiable, 
the Taylor series of each ${\Psi}_i$ converges in a neighborhood of
every $\x \in U$,  
and the derivative mapping $d_{\x}{\Psi}: \R^d \to \R^k$ has
  rank $d$.
By a  {\sl  $d$-dimensional real analytic submanifold {of}
  $\R^k$} we mean a subset $Y \subset \R^k$ such that for
every $\y \in Y$ there is a neighborhood $V
\subset \R^k$ containing $\y$, an open set $U \subset \R^d$, 
and a real analytic immersion ${\Psi}: U \to \R^k$ such that
${V} \cap {Y} = {\Psi}({U}).$} 

\smallskip
The crucial property, which distinguishes real analytic submanifolds from
smooth manifolds and follows easily from definitions, 
is the
following

	\begin{lemma}\label{lem: crucial} 
	{Let $L$ and $R$ be real analytic submanifolds of $\R^k$, equipped with
the topology they inherit as subsets of $\R^k$. Suppose that $L$ is
connected, $R$ is closed, and $L \cap R$ has nonempty interior in
$L$. Then $L \subset R$. }
   \end{lemma}

{To make the paper self-contained we provide the}

\begin{proof}
  {Let
  $d_1,  d_2$ denote the dimensions of $L$ and $R$. Since $L\cap R$
  has nonempty interior in $L$, and by invariance of dimension under
  immersions, we have $d_1 \leq d_2$. Let $W$ denote the interior (in
  $L$) of $L \cap R$. If $W$ is closed (in $L$), then by connectedness
  of $L$ we have $L = W$, and hence $L \subset R$ and there is nothing
  to prove. We will assume that $W$ is not closed in $L$ and derive a
  contradiction.}

{Let $p \in \overline{W} \sm W$. Since the closure is taken in $L$ we
  have $p \in L$, and since $R$ is closed we have $p \in R$. By the
  real analyticity of $L$ and $R$,
 there are open subsets
  $U_1 \subset \R^{d_1}, \, U_2 \subset \R^{d_2}$ and real analytic
  immerions $\Psi_i: U_i \to \R^k$ such that $\Psi_1$ parameterizes a
  neighborhood of $p$ in $L$, and $\Psi_2$ parameterizes a neighborhood
  of $p$ in $R$. For $i=1,2$, let $p_i \in U_i$ such that $\Psi_i(p_i) = p$, and let
  $W_i = \Psi_i^{-1}(W)$. Then $p_1$ belongs to the closure of $W_1$
  in $U_1$, and $W_1$ is open in $U_1$; thus by the uniqueness of
  analytic continuation (see \cite[\S 1.2]{Krantz} or \cite[\S
  VI.6]{KN}), $\Psi_i$ is uniquely  
  determined in a neighborhood of $p$ by $\Psi_1|_{W_1}$. In
  particular, $L$ is uniquely determined in a neighborhood of $p$ by
  $W$.} 
%
%

By the inverse function theorem for real analytic
  immersions (see \cite[\S 1.8]{Krantz}), there is a real analytic
  inverse 
  $\Psi_2^{-1} $ to $\Psi_2$. 
  {Abusing notation, denote by $\R^{d_1}$ the  coordinate plane in
  $\R^{d_2}$ defined as
  $$\R^{d_1} \df  \{(x_1, \ldots, x_{d_2} )\in
\R^{d_2} : x_{d_1+1} = \cdots = x_{d_2}=0\}.$$
By making $U_2$ smaller and by replacing $U_2$ with its image under a real analytic
diffeomorphism, we can assume that 
  $W_2$ is 
  open 
  in $U_2 \cap \R^{d_1}$; indeed, 
to see that such a change of variables exists, see \cite[Proof of
Thm.\ 1.9.5]{Krantz}.
Once again, using the uniqueness of analytic continuation
of $\Psi_2|_{W_2}$, we see 
that $\Psi_2|_{U_2 \cap \R^{d_1}}$ is
determined by $\Psi_2|_{W_2}$. Since the two analytic continuations
agree, we see 
 that $p$ belongs to the interior of $L \cap R$.} 
\end{proof}

{Arguing as in the above proof of Theorem \ref{thm: classical 
   theorem}(a), we arrive at}
        \begin{theorem} \label{cor: matrices}
	Let $m,n\in\N$ with $n > 1$, {and let a non-increasing function
        $f: \R_{>0} \to \R_{>0}$  and an
        arbitrary countable  
        collection $\{R_\ell\}$ of proper {closed} analytic submanifolds     of $\mr$ be given}. Then
        $\UA_{m,n}(f)  \smallsetminus \cup_{\ell} R_\ell$	 
	is uncountable and dense in $\mr$. 
	\end{theorem}
	
	{Likewise, for any  weight vector $    \vw = (\va,\vb) \in \R_{>0}^{m+n}$ as in \S\ref{subsec: weights intro}, in view of    \equ{reduction}, the conclusion of the above theorem holds verbatim for the set $\UA_{m,n}(f, \vw) $.}
\ignore{Concerning weighted approximation as in \S \ref{subsec: weights
  intro}, for a fixed $\vw$ one can use Theorem \ref{thm: general
  axiomatic theorem} to  
   construct uncountably many nontrivial elements of
   $\UA(f, \vw)$.} 
{{Moreover, one can take} a subset $\mathcal{W}$ of
   $\R_{>0}^{m+n}$ and say that  \amr\ is $(f, \mathcal{W})$-\textsl{uniform}, denoted
   with some abuse of notation by $A\in \UA_{m,n}(f, \mathcal{W})$, if  
for every sufficiently
large $t$  one can find
$\vq  \in \Z^n\nz$ and $\vp \in \Z^m$
 such that  for any ${\vw = (\va,\vb)}\in \mathcal{W}$ {the inequalities
\equ{diweighted} hold}. Clearly this is 
stronger than being $(f, \vw)$-{uniform} for any $\vw\in
\mathcal{W}$. 
The next statement
follows from {Theorem \ref{thm: special axiomatic theorem}}   as easily
as the previous one did. It immediately implies {both parts of} Theorem \ref{thm:
  intersection with weights}.
 \begin{theorem} \label{cor: weighted matrices}
	Let $m,n\in\N$ with $n > 1$.  Suppose that for any $k\in\N$ we are given a
        non-increasing function   $f_k: \R_{>0} \to \R_{>0}$  and 
       {a subset $\mathcal{W}_k$ of $  \R_{>0}^{m+n}$ such that}
       \eq{afw}{{\sup_{ (\va,\vb)\in \mathcal{W}_k}\max_{i}\alpha_i <\infty\ \ \text{ and }\inf_{(\va,\vb)\in \mathcal{W}_k}\min_{j}\beta_j > 0.}}Then  for any  countable  
        collection $\{R_\ell \}$ of proper  analytic
        submanifolds of $\mr$, the set
        $$\bigcap_k \UA_{m,n}(f_k, \mathcal{W}_k)
        \smallsetminus \bigcup _{\ell} R_\ell$$ 	
	is  {uncountable and} dense in $\mr$. In particular,  the set
        $$\bigcap_{\vw\in\R_{>0}^{m+n}} \UA_{m,n}(f, \vw)  \smallsetminus \bigcup
        _{\ell} R_\ell$$ is  {uncountable and} dense in $\mr$ for any  $\{R_\ell\}$ as above
        and any non-increasing positive function $f$.  
	\end{theorem}
        \begin{proof} Again we take $Y = \mr = X_k$, $ \I =  \Z^m\times (\Z^n\nz)$, and $\varphi_k = \Id$ for each $k$,
        {and, to define}         
        the Diophantine systems associated with each $X_k$, 
        {let}
        \begin{equation}\label{eq: Shreyasi and Nattalie} h_{k,(\vp,\vq)} \df \sup_{\vw\in \mathcal{W}_k}\|\vq\|_\vb\quad \text{
          and}\quad d_{k,(\vp,\vq)}(A) \df \inf_{\vw\in \mathcal{W}_k}\|A\vq -
        \vp\|_\va.
        \end{equation}
Because of 
{\equ{afw}}, for any $(\vp,\vq)$ and any $k$ 
{the value $h_{k,(\vp,\vq)}$ is finite} and 
{the function $d_{k,(\vp,\vq)}$ is continuous}.
 Thus the same argument, in view of {Remark \ref{doesnotmatter}(2)}, yields the proof of the first part of the theorem.
     And for the `in particular' part one simply writes $\R_{>0}^{m+n}$ as the union of countably many 
     subsets {$\mathcal{W}_k$
     satisfying \equ{afw}.}
        \end{proof}}

\section{{Uniformly approximable row vectors \\ on analytic submanifolds and some fractals}}\label{sbmfolds}


Our next goal is to apply Theorem \ref{thm: special axiomatic theorem} to study uniform approximation on submanifolds $Y$ of $\mr$. It is clear that the method we are using in this paper (which is essentially Khintchine's original argument) is not applicable if $\dim Y = 1$; and indeed, there are very few results in the literature dealing with singular vectors on curves, with many open questions, {see e.g.\  \cite{PR} and references therein}. However the method does work if $m=1$, $Y$ is connected analytic submanifold of $\R^n\cong M_{1,n}(\R)$,  and the dimension of $Y$ is at least $2$, see Theorem \ref{cor: linear forms manifolds simple}.

In this section we will generalize the aforementioned theorem. 
But first let us observe that the situation gets more complicated when $\min(m,n) > 1$. The next proposition gives an example of a three-dimensional submanifold (in fact, an affine subspace) of $M_{2,2}(\R)$ which does not contain any $f$-uniform matrices if  $f$ decays rapidly enough. 
{Let us recall that a real number  $\alpha$ is badly approximable if \eq{ba}{
\inf_{q\in \mathbb{Z}\nz}
 {|q|} \dist(\alpha q, \Z)
>0.}}

\begin{proposition}\label{kolya's example}
Let $\alpha\in \mathbb{R}$ be a badly approximable number, let $\lambda = \frac{\sqrt{5}+3}{2}$, and let   $Y$ be the set of $2\times 2$ matrices of the form $ A = \left(
\begin{array}{cc}
\alpha &* \cr
* & *
\end{array}
\right)$. Then $
\hat{\omega}(A)\le  \lambda$ for any {totally irrational} $A\in Y$; that is, no  {totally irrational} $A\in Y$ is $f$-uniform as long as $f$ decays faster than {$\phi_{\lambda+\vre}$ for some $\vre>0$}. 
\end{proposition}

\begin{proof} We need to recall the notion of an \textsl{ordinary Diophantine exponent}, which is an asymptotic version of \equ{omegahat}; namely, $${\begin{aligned}{\omega}(A) &\df \sup
  \left\{a> 0 \left| 
 \begin{aligned}
 \text{ for an unbounded set of }
  t > 0 \quad\\
  \ \exists \,\vq \in \Z^n\nz,\ \mathbf{p} \in \Z^m
  \qquad \\
   \text{
    such that }\equ{digeneral} \text{ holds for 
    $f=\phi_a$}\end{aligned}\right.\right\}\\
    &= \inf
  \left\{a> 0 : 
 \inf_{\q\in \mathbb{Z}^n\nz} 
 \|\q\|^a  \dist(A\q,\Z^m) 
>0  \right\}, 
\end{aligned}
}$$
where the distance is computed using the supremum norm on $\R^m$.
{It follows from \equ{ba}} that
$\omega(\alpha) \le 1$. Take an arbitrary $\beta\in\R$ {such that $1,\alpha,\beta$ are linearly independent over $\Q$}, and consider $\v =  \left(
\begin{array}{cc}
\alpha  \cr
\beta
\end{array}
\right)$. Clearly we have $$
\inf_{q\in \mathbb{Z}\nz}
 q \dist(\v q, \Z^2)
>0;$$ hence $\omega(\v) \le 1$ as well.

Now we will use the inequalities due to Jarn\'ik 
relating ordinary and uniform exponents of $\v$ and $\v^T$:
\begin{equation}\label{j1}
\omega(\v) \ge \frac{\hat{\omega}(\v)^2}{1-\hat{\omega}(\v)}
\end{equation}
and
\begin{equation}\label{j2}
\hat{\omega}(\v) + \frac{1}{\hat{\omega}(\v^T)} = 1
\end{equation}
{(see \cite{Jarnik38} and \cite{Jarnik54} respectively)}.
It follows from (\ref{j1}) that 
$
\hat{\omega}(\v)\le \frac{\sqrt{5}-1}{2},
$
and from (\ref{j2}) we can obtain
$
\hat{\omega}(\v^T) \le  \frac{\sqrt{5}+3}{2} = \lambda.
$
Thus for any $A\in M_{2,2}(\R)$ with $(\alpha,\beta)$ as its row vector we have 
$
\hat{\omega}(A)\le \hat{\omega}(\v^T) \le \lambda.
$
\end{proof}

Our next result shows that such examples are impossible when $m=1$ and the dimension $d$ of a submanifold $Y$ of  $\R^n\cong M_{1,n}$ is at least $2$.  Further, we will exhibit vectors in $Y$ which are $f$-uniform of order $d-1$, as defined in \S\ref{mflds}.

  \begin{theorem} \label{cor: linear forms manifolds}
	Let $n\ge 2$, let $Y$ be a $d$-dimensional connected analytic submanifold of $\R^n$, where $d\ge 2$, and let $1\le g \le d-1$.  Suppose we are given an arbitrary non-increasing function
        $f: \R_{>0} \to \R_{>0}$  and a   countable  
        collection $\{R_\ell \}$   of proper {closed} analytic submanifolds of $Y$. Then the set
        \eq{complement}{Y\cap \UA_{1,n}(f;g)  \smallsetminus \bigcup_{\ell\in\N} R_\ell}
	is uncountable and dense in $Y$. In particular, if one in addition assumes that $Y$ is 
	not contained in any proper rational affine subspace of $\R^n$, then the set
        $Y\cap \UA^*_{1,n}(f;g)
       $ is uncountable and dense in $Y$.
\end{theorem}
	 

To prove  Theorem \ref{cor: linear forms manifolds} we first need to describe the Diophantine system responsible for approximation of order $g$. Unwrapping Definition \ref{orderg}, we can easily see that  $\x\in\R^n$ is $f$-uniform of order $g$ if and only if it  is $(\mathcal{X},f)$-uniform,  where $\mathcal{X}$ is the 
Diophantine  {system} given by
 {\eq{systemorderg}{
 \begin{aligned}X = \R^n,\ \ &\I =  \Z^g \times \big((\Z^n)^g\smallsetminus\{\text{linearly dependent $g$-tuples}\}\big),\\ 
&\qquad h_{p_1, \dots, p_g,\mathbf{q}_1, \ldots,
  \mathbf{q}_g} = \max_{i=1,\dots,g}\| \mathbf{q}_i \| ,\\ \ &d_{p_1, \dots, p_g,\mathbf{q}_1, \dots,
  \mathbf{q}_g}(\x) =  \max_{i=1,\dots,g} | \mathbf{q}_i \cdot \mathbf{x} -p_i| {.}\quad\end{aligned}}
Note that the zero locus of $d_{p_1, \dots, p_g,\mathbf{q}_1, \dots,
  \mathbf{q}_g}$ is precisely the intersection of $g$ rational affine hyperplanes $L_{p_i, \vq_i}$, which is a codimension $g$ rational affine subspace of $\R^n$; moreover, any rational affine subspace of $\R^n$ of codimension $g$ can be (in many different ways) written as such an intersection. Consequently, the collection $\mathcal L$ of all  rational affine subspaces of $\R^n$ of codimension $g$ is aligned with $\mathcal{X}$ as in \equ{systemorderg}. It is trivial to check that $\mathcal L$ is totally dense and respects any closed analytic submanifold of $\R^n$. Hence the case $Y = \R^n$ of Theorem \ref{cor: linear forms manifolds} easily follows from Theorem \ref{thm: special axiomatic theorem}. Let us now prove the general case where $Y$ is an arbitrary $d$-dimensional analytic submanifold of $\R^n$.
  For the proof  we will need the following




	
	
\begin{lemma}\label{localtd} Let $Y$ be a $d$-dimensional $C^1$ {embedded} submanifold of $\R^n$, where $d\ge 2$, and let $1\le g \le d-1$.  Then for
any $\y \in Y$ there exists an open subset $W$ of $\R^n$ containing  $\y$ and a totally dense (in $Y\cap W$) collection $\mathcal{L}_{W}$ consisting of 
intersections of rational affine subspaces of $\R^n$ of codimension $g$ with $Y\cap W$.
\end{lemma}

\begin{proof} For
any $\y \in Y$ one can choose a neighborhood $V
\subset \R^n$ containing $\y$, an open set $U \subset \R^d$, 
and a differentiable {embedding} ${\Psi}: U \to \R^n$ such that
${V} \cap {Y} = {\Psi}({U})$. Let $\x\in U$ be such that ${\Psi}(\x) = \y$; then $d_{\x} {\Psi}$ has rank $d$, and by permuting coordinates in $\R^n$ without loss of generality we can assume that the leftmost $d\times d$ submatrix of $d_{\x} {\Psi}$ is non-singular. Then, in view of the implicit function theorem, we can choose an open  subset $W$ of $В$ containing $\y$ such that
   $${
W\cap Y =\big \{\big(x_1, \dots,x_d, {\psi}(x_1, \dots,x_d)\big) : (x_1, \dots,x_d )\in
O\big\},
}$$
where $O$ is an  open ball in  $\R^d$ and $\psi:
O \to \R^{n-d}$ is a differentiable function. 

Now define
$$
\mathcal{L}_{W} \df \big \{L_M\, \cap\, W\,\cap\, Y : M \text{ is a rational affine subspace  of }\R^d\text{ of codimension }g
\big\},
$$
where  
\eq{lm}{
L_M \df  \big\{(x_1, \dots,x_n):  (x_1, \dots,x_d)\in M 
\big\}.
}
Equivalently,  $L_M \cap W\cap Y$ can be written as
 \eq{nicerform}{
\big \{\big(x_1, \dots,x_d, \psi(x_1, \dots,x_d)\big) : (x_1, \dots,x_d )\in
M\cap O\big\}.
}
Since $d\ge 2$ and $1\le g \le d-1$, the collection 
$$\big \{M \cap O : M \text{ is a rational affine  subspace  of }\R^d\text{ of codimension }g
\big\}
$$
is totally dense in $O$, which immediately implies that  $
\mathcal{L}_{W}$ is  totally dense in $Y\cap W$. 
\end{proof}

	\begin{proof}[Proof of Theorem \ref{cor: linear forms manifolds}]   
Our goal is to apply 	Theorem \ref{thm: special axiomatic theorem} locally in the neighborhood of any given point   $\y\in Y$. 
{Choose  an open subset $W$ of $\R^n$ containing  $\y$ and a collection $\mathcal{L}_{W}$ 	as in the above lemma.}	Since $L_M$ {as in \equ{lm}} is a codimension $g$ rational affine subspace of $\R^n$, it is clear that $\mathcal L$ is aligned with the Diophantine system \equ{systemorderg}. Now let us recall that $Y$ was assumed to be an analytic manifold; hence the map $\psi$ is real analytic, and therefore every $L_M \cap W\cap Y$ is a connected analytic submanifold of $W\cap Y$. If $R$ is a   {closed} analytic submanifold of $Y$ such that $L_M \cap W\cap  R$ has non-empty interior in $L_M \cap W$, it follows that $L_M\subset W\cap  R$, which, in view of Lemma \ref{lem: crucial}, implies that
$L_M \subset R$. Hence $
\mathcal{L}_W$ respects $R$, and we can apply Theorem \ref{thm: special axiomatic theorem}(b) and conclude that 
       the intersection of the set \equ{complement} with 
       $W$ is uncountable, finishing the proof of the theorem. 
       \end{proof}

          Let us now show that the same method can produce uniformly approximable vectors on certain fractals. Here is an example from \cite{KMW}:
    
    \begin{theorem}\label{thm: fractals}
Let $n \geq 2$ and let $Y_1, \ldots, Y_n$ be perfect subsets
of $\R$ such that 
\eq{density}{\Q \cap Y_k \text{ 
is dense in }Y_k\text{  for each }k \in \{1,2\}.} Let $Y = \prod_{j=1}^n Y_j$. Suppose we are given an arbitrary non-increasing function
        $f: \R_{>0} \to \R_{>0}$  and a   countable  
        collection $\{R_\ell \}$   of proper {closed} analytic submanifolds of $\R^n$. Then 
        $$Y\cap \UA_{1,n}(f
        )  \smallsetminus \bigcup_{\ell\in\N} R_\ell$$	 
	is uncountable and dense in $Y$. In particular, 
        $Y\cap \UA^*_{1,n}(f
        )
       $ is uncountable and dense in $Y$. 
\end{theorem}
      
 
\begin{proof} 
Let $\mathbf{e}_1, \ldots, \mathbf{e}_n$ be the standard basis
vectors in $\R^n$, and let
$\{A_i\} $ be the collection of all rational affine 
hyperplanes of $\R^n$ 
which are normal
to one of $\mathbf{e}_1 , \mathbf{e}_2$ and have nontrivial
intersection with $Y$; that is, each of the 
$A_i$ is of the form
$${
A_i = \left\{ \x \in \R^{n} : x_{k_i} = r_i \right\},
\text{ where } r_i \in \Q \text{ and } k_i \in \{1,2\};}$$
note that necessarily we have $r_i\in {Y_{k_i}}$.
We claim that the collection  $\{ Y \cap A_i\}$, which is obviously aligned with $\mathcal{X}_{1,n}$, is totally dense.
Indeed, \equ{vd1} clearly follows from \equ{density}. To prove \equ{vd3}, take an open subset $W$ of $\R^n$ of the form $I_1\times\cdots\times I_n$, where $I_j$ are open intervals in $\R$, and suppose that \eq{a}{A = \left\{ \x \in \R^{n} : x_1 = r \right\}} intersects with $Y$ non-trivially, that is, we have $r\in Y_1$. 
Then, again in view of \equ{density}, the union of subspaces $A_j$ such that $A_j\cap A \cap Y\cap W \ne\varnothing$ is dense in $(I_1\cap Y_1)\times\R\times \cdots\times \R$, hence the union of corresponding intersections $Y \cap A_j$ is dense in $Y\cap(I_1\times\R\times \cdots\times \R)$.

It remains to prove that the collection $\{ Y \cap A_i\}$ respects  an arbitrary closed analytic submanifold $R$ of $\R^n$. Suppose that, for $A$ as in  \equ{a},  {the intersection} $Y\cap A \cap R$ has a non-empty interior in  $Y\cap A$. 
 {Take a point $\x = (x_1,\dots,x_n)$ in the interior of $Y\cap A \cap R$; then there exists a neighborhood $W$ of $\x$ in $\R^n$ such that $Y\cap A \cap R$ contains $\big(\{x_1\} \times Y_2 \times ... \times Y_n\big)\cap W$.
Since each of the sets $Y_i$ is perfect,  there are sequences in $Y\cap A \cap R$ approaching $\x$ from each coordinate direction. This makes it possible to determine all partial derivatives of all orders of any analytic function on $A$  at  $(x_1,\dots,x_n)$  by its values on $Y\cap A \cap R$. In particular, since $R$ is a closed analytic manifold, it has to contain $A$, hence it is respected by $Y\cap A$.}

{Now we can conclude that}  Theorem \ref{thm: special axiomatic theorem} applies and implies the density of $Y\cap \UA_{1,n}(f)  \smallsetminus \bigcup_{\ell\in\N} R_\ell$   for any countable set of proper closed analytic submanifolds $R_\ell$ of $\R^n$. The uncountability  follows as well, since the sets $\{ Y \cap A_i\}$ have no isolated points.  
\end{proof}

 \begin{remark}\label{Koch} \rm
     Our arguments can be adapted to many other fractal sets. We sketch
     one more example, involving a `rational Koch snowflake' which can
     be treated by our method. To define it, let $\alpha, \beta, \gamma, \delta$ be
     positive rational numbers, where $\alpha < \beta < \delta < 1$,
     and let  $Y$ be the attractor of the iterated function system
  {$\{\varphi_1,\varphi_2, \varphi_3, \varphi_4\}$, where $\varphi_i : \R^2 \to \R^2$} is the unique 
  orientation preserving contracting similarity map sending the
 points {
$\x_0 = (0,0)$ and $\y_0 =
  (1,0)$}, to the 
  points {$\x_i, \y_i$}  defined by
  {\[\begin{split}
       \x_1 = (0,0),  \ \ \ &\y_1 = ( \alpha, 0)  \\
        \x_2 = \y_1,  \ \ \ \ \ \ &\y_2 = ( \beta, \gamma)  \\
       \x_3 = \y_2,   \ \ \ \  \ \ &\y_3 = (\delta, 0)  \\
       \x_4 = \y_3, \ \ \ \  \ \ &\y_4 = ( 1, 0) .
    \end{split}\]}
Then $Y$ is a topological curve  (continuous
image of a segment) of Hausdorff dimension greater than one. The usual
version of the Koch snowflake  has a similar description, with {$\alpha = \frac13,\,
\beta =  \frac12,\, \gamma =  \frac{\sqrt{3}}{6}, \delta =  \frac23$ (see 
{\cite
[Figure 0.2]{Falconer}})}.

For our application, it is important to note that
if $\theta$ is the slope of the line from {$\x_2$ to $\y_2$}, then $Y$ 
contains a dense collection of rational points $Y_0$, and the horizontal
lines and lines of slope $\theta$ through points of $Y_0$ intersect $Y$
in an uncountable perfect set. This property makes it possible to
repeat the proof of Theorem \ref{thm: k singular stronger fractals}, taking the {sets} $\{A_i\}$ to be the
lines {having} these two slopes and passing through $Y_0$. 
\end{remark}

}

\section{{Vectors that are  $k$-singular for all $k$}}\label{sec: k singular}
In this section we prove {a strengthening of} Theorem \ref{thm: k singular all k}. 
To state it, let us consider a definition generalizing the notion of $k$-singular vectors. Namely, for a   non-increasing $f:\R_{>0} \to \R_{>0}$ say that $\mathbf{x}\in\R^n$ is \textsl{$f$-uniform of degree $k$} 
 if for any sufficiently large $t$ there exists a
polynomial $P \in \Z[X_1, \ldots, X_n]$ such that $\deg (P) \le k, \
H(P) \leq t,$ and $|P(\mathbf{x})| \leq f(t)$. Clearly $\x$, {viewed as a row vector,} is $k$-singular {(as defined in \S\ref{ksing})} if and only if it is  $\varepsilon\phi_{N_k}$-uniform of degree $k$, where $N_k$ is as in  \eqref{eq: def n}.
Also it is clear that $k$-algebraic vectors ({again see \S\ref{ksing} for a definition}) are $f$-uniform of degree $k$ for any positive $f$.}

{Let us} say that a submanifold  ${Y} \subset {\R^n}$ is 
{\em transcendental} if ${Y}$ is not 
contained in an algebraic  subvariety defined over $\Q$; more concretely, if there
is no nonzero polynomial in {$\Z[X_1, \ldots, X_n]$} which vanishes on
${Y}$.

\begin{theorem}\label{thm: k singular stronger}
  Let {$n \geq 2$, let  $Y  $} be a
  transcendental 
{analytic 
 submanifold of $\R^n$} {of dimension  $d\ge 2$. Suppose that for each $k\in\N$ we are given a non-increasing  $f_k:\R_{>0} \to \R_{>0}$.  Then there is a dense  uncountable set of $\mathbf{x} \in {Y}$
which are not algebraic and are $f_k$-uniform of degree $k$ for all $k \in \N$. In particular,  $Y$ contains a dense uncountable set of}
$\mathbf{x}
 $ which are not algebraic and are 
  $k$-singular for all $k \in \N$. 
  \end{theorem}

This result extends   {\cite[Theorem 1.7]{KMW}; namely, 
there} it was shown
that 
$Y$  {as above}  contains uncountably many singular vectors. 

\smallskip
{The proof is based on the following observation. {For
$\x \in \R^n$, denote by $\varphi_k(\x)$ the vector  with coordinates consisting of all non-constant monomials of degree at most $k$ in $n$ variables  $x_1,\dots,x_n$, viewed as a row vector in $\R^{N(k,n)}$, and for each $k$ consider the standard Diophantine system {$\mathcal{X}_{1,N(k,n)}$} on $X_k = \R^{N(k,n)}\cong M_{1,N(k,n)}$ with distance functions from rational affine hyperplanes and standard heights. 
Then it is clear that $\mathbf{x} \in \R^n$ is $f_k$-uniform of degree $k$ if and only if $\varphi_k(\mathbf{x}) \in \UA_{1,N(k,n)}(f_k)$.} Furthermore we have the following}

  \begin{proof}  [Proof of Theorem \ref{thm: k singular stronger}]
Take $Y$ as in the statement of Theorem \ref{thm: k singular stronger} and 
argue as in {the proof of} Theorem~\ref{cor: linear forms manifolds}. {Namely, for any  $\y\in Y$ define 
\begin{itemize}
\item[$\bullet$] a neighborhood $W$ of $\y$ such that $Y\cap W$ is a graph of an analytic function $\R^d\to \R^{n-d}$, and
\item[$\bullet$] (using  Lemma \ref{localtd} with $g=1$) a  totally dense  collection $\mathcal{L}_W$ of intersections of  some rational affine hyperplanes of $\R^n$ with $Y\cap W$.
\end{itemize}}

{It is clear} that $\mathcal{L}_W$ {as above} is aligned with $\mathcal{X}_{1,N(k,n)}$ via $\varphi_k$: {indeed, if 
$\pi: \R^{N(k,n)} \to \R^n$ denotes the projection onto the first $n$ coordinates and $L$ is a rational affine hyperplane in $\R^n$,  then $\varphi_k(L)$ is contained in  $\pi^{-1}(L)$, which  is a rational affine hyperplane in $ \R^{N(k,n)}$.
Now} let  $\{R_\ell\}$ be the collection of 
sets of the form 
$\{\y \in Y: P(\y) =0\}$, where $P$ ranges over all nonzero
polynomials in $\Z[X_1, \ldots, X_n]$. 
Since $Y$ is assumed to be transcendental, it follows that each of the sets $R_\ell$ is a proper subset of $Y$. It is easy to  show that $\mathcal{L}_W$ respects $R_\ell$: indeed, any element  $L$ of  $\mathcal{L}_W$ is of the form \equ{nicerform} {with analytic ${\psi}$; furthermore,} 
 the fact that $L\cap R_\ell$ has non-empty interior in $L$ implies that 
$P\big(x_1, \dots,x_d, \psi(x_1, \dots,x_d)\big) = 0$ on an open subset of an affine hyperplane $M$ of $\R^d$. Since $d\ge 2$ and ${\psi}$ is analytic, it follows that $L\subset R_\ell$. 
Hence  Theorem \ref{thm: special axiomatic theorem} applies, and we conclude that  the set  $${\bigcap_{k} \varphi_k^{-1}\big(\UA_{1,N(k,n)}(f_k)\big)
        \smallsetminus  \bigcup_{{{\ell}}} R_{{\ell}}},$$
	which contains the set of vectors  that are not algebraic and are $f_k$-uniform of degree $k$ for all $k \in \N$, is uncountable and dense in $Y$ for any choice of non-increasing 
     functions $f_k: \R_{>0} \to \R_{>0}$.
     \end{proof}

 
Using similar ideas and arguing as in the proof of Theorem \ref{thm: fractals} and Remark \ref{Koch}, one can establish a  similar result for fractals: 
\begin{theorem}\label{thm: k singular stronger fractals}
Let $Y$ be either the product of 
perfect subsets $Y_1, \ldots, Y_n$
of $\R$ such that 
\equ{density} holds ($n\ge 2$), or a rational Koch snowflake ($n=2$). 
Then the conclusions of Theorem \ref{thm: k singular stronger} hold for $Y$.
	 \end{theorem}

\section{Sumsets of  sets of totally irrational \\ uniformly approximable vectors}\label{sumsets}
Our goal in this section is to prove Theorem \ref{thm: schleischitz extension}, stating that when $n\ge 3$, the sum of
{$\UA^*_{1,n}(f_1) $ with $\UA^*_{1,n}(f_2) $ is $\R^n$ for any pair of approximating functions $f_1,f_2$. Clearly, by replacing each of the functions $f_i$ with $\min(f_1,f_2)$,
it is enough to take a non-increasing $f:\R_{>0}\to \R_{>0}$ and prove that $\UA^*_{1,n}(f) +\UA^*_{1,n}(f) =\R^n$.}
Equivalently, since $\UA^*_{1,n}(f) $ coincides with $ -\UA^*_{1,n}(f)$, we need to show that  for any fixed  $\mathbf{z} \in \R^n$   the intersection of 
$\UA^* _{1,n}(f) $ and its translate $\UA^*_{1,n}(f) - \mathbf{z}$ is not empty. 
Unwrapping the definitions, we see that it amounts to finding  $\x\in\R^n$ such both $\x$ and $\x +\z$ are  totally irrational, and such that  for every sufficiently
large $t$  one can find
$\vq  \in \Z^n\nz$ and $p \in \Z$
with
$$|\x\cdot\vq -p|\le f(t)
  \ \ \ \mathrm{and}  
\ \ \|\vq\|\le t,$$
and also $\vq'  \in \Z^n\nz$ and $p' \in \Z$
with
$$|(\x-\z)\cdot\vq' -p'|\le f(t)
  \ \ \ \mathrm{and}  
\ \ \|\vq'\|\le t.$$
Rewriting the two systems of inequalities above   as
$$\max(|\x\cdot\vq -p|,|(\x-\z)\cdot\vq' -p'|) \le f(t)
  \ \ \ \mathrm{and}  
\ \ \max(\|\vq\|,\|\vq'\|)\le t,$$
we see that the problem reduces to a new Diophantine system 
 $${\mathcal{X}_\z := \big( \R^n,  \{d_{p,p',\vq,\vq'}\},  \{h_{p,p',\vq,\vq'}\}\big)}$$
where $d_{p,p',\vq,\vq'}(\x) := \max(|\x\cdot\vq -p|,|(\x-\z)\cdot\vq' -p'|)$, $h_{p,p',\vq,\vq'} := \max(\|\vq\|,\|\vq'\|)$, 
and $(p,p',\vq,\vq')$ runs through $\Z\times \Z\times (\Z^n\nz)\times(\Z^n\nz)$.}  Note that the zero locus of $d_{p,p',\vq,\vq'}$ is precisely the intersection $L_{p, \vq} \cap (L_{p', \vq'} + \z)$ of two hyperplanes in $\R^n$; in other words, the collection 
\begin{equation*}{
\mathcal{L}_\z\df \left\{L_{p, \vq} \cap (L_{p', \vq'} + \z)  \left| \begin{aligned}\ p,p'\in\Z,\ \vq,\vq'\in\Z^n\nz,\ \ \\ \ \vq\text{ and }\vq' \text{ are not proportional}\end{aligned}\right.\right\}}
\end{equation*}
is aligned with the Diophantine system ${\mathcal{X}_\z}$. 

\smallskip
The crucial step of proof of Theorem \ref{thm: schleischitz extension} will be the following
\begin{lemma}\label{thm: intersection TD} Let   $n \geq 3$ and $\z\in\R^n$. Then $\mathcal{L}_\z$
is totally dense relative to the collection of all affine hyperplanes in $\R^n$.
\end{lemma}
Note that we are not able to prove the stronger form of total density, that is, the validity of \equ{vd3} for $\mathcal{L} = \mathcal{L}_\z$. 

\begin{proof}[Proof of Lemma \ref{thm: intersection TD}]
{Let us start with the following elementary
\begin{sublemma}\label{sublemma: TD} Let   $n \geq 3$ and let 
$N$ be an affine hyperplane of $\R^n$.
Then the collection 
$$\mathcal{L}(N)\df \{M\cap N: M \text{ is a rational affine hyperplane of }\R^n,\ M\ne N\}$$
is totally dense in $N$; furthermore, the union in \equ{vd3} is dense in $N$.
\end{sublemma}
\begin{proof} By permuting coordinates, without loss of generality we can express $N$ in the form
$$
x_n = a_1x_1 + \cdots+a_{n-1}x_{n-1} + a_n,\quad\text{where }a_i\in\R;
$$
this way the projection $\pi: (x_1,\dots,x_n)\mapsto(x_1,\dots,x_{n-1})$ maps $N$ bijectively onto $\R^{n-1}$. Further, 
any $L\in \mathcal{L}(N)$ is of the form
\eq{LN}{
\{\x\in\R^n: x_n = a_1x_1 + \cdots+a_{n-1}x_{n-1} + a_n, \ q_1x_1 + \cdots+q_{n}x_{n} = p\}
}
for some $\q\in\Z^n\nz$ and $p\in\Z$. It suffices to prove that the collection $\pi\big(\mathcal{L}(N)\big)$ of affine hyperplanes of $\R^{n-1}$  is totally dense in $\R^{n-1}$. By considering the case $q_n = 0$ 
in  \equ{LN}, 
it is easy to see that $\pi\big(\mathcal{L}(N)\big)$ contains the collection of all rational affine hyperplanes of $\R^{n-1}$.
The latter collection, in additional to being totally dense, has the following property: for any affine hyperplane $L$ of $\R^{n-1}$, any open $W\subset \R^{n-1}$ with $W\cap L \ne\varnothing$ and any $\y\in\Q^{n-1}\smallsetminus L$ there exists $\y'\in\Q^{n-1}$ such that the (rational) line passing through $\y$ and $\y'$ intersects $W\cap L$. This clearly implies that the union of all rational affine hyperplanes touching $W\cap L \ne\varnothing$ is dense in $\R^{n-1}$, hence  the total density of $\pi\big(\mathcal{L}(N)\big)$.
\end{proof}
The above sublemma in particular implies that for any $N = L_{p', \vq'} + \z$ the union $$\bigcup_{p\in\Z,\,\vq\in\Z^n\setminus\R\vq} L_{p, \vq}\cap N$$ is  dense in $N$. Since the union  of all 
 rational affine
hyperplanes translated by $\z$ is dense in $\R^n$, it follows that $\bigcup_{L\in{\mathcal L}_\z}L$ is dense in $\R^n$, i.e.\ condition \equ{vd1} holds.}


Now take $L\in{\mathcal L}_\z$, that is, $L  = L_{p, \vq} \cap (L_{p', \vq'} + \z) $ such that $ \vq$ and $\vq'$  are not proportional, and {choose} an open subset $W$ of $\R^n$ {with} $L\cap W\ne\varnothing$.  For brevity denote  $M  \df L_{p, \vq}$ and ${N} \df L_{p', \vq'} + \z$. {It follows from Sublemma \ref{sublemma: TD}  that the union of $P\cap N$ over all  rational affine hyperplanes $P\ne N$ that satisfy $P\cap W\cap L\ne\varnothing$ is dense in $N$. Applying a translation by $-\z$ to the above conclusion one gets that the union of $(P+\z)\cap M$ over all  rational affine hyperplanes $P $ with $P+\z \ne M$ and $P\cap W\cap L\ne\varnothing$ is dense in $M$.} It follows that the closure of the union of all $L'\in{\mathcal L}_\z$ such that $L'\cap W\cap L\ne\varnothing$ contains $M\cup N$, thus {it} cannot be a subset of a single affine hyperplane.
    \end{proof}

{Now,} since $\mathcal{L}_\z$ obviously respects any affine subspace in $\R^n$, 
we can apply Theorem \ref{thm: special axiomatic theorem} {to $\mathcal X_\z$
with 
$\mathcal{R} = \{R_\ell\}$ being an arbitrary countable collection of affine hyperplanes of $\R^n$. This immediately yields the proof of Theorem \ref{thm: schleischitz extension} in a slightly more general form, with $\UA^*_{1,n}(f_k)$, $k=1,2$, in \equ{sumset} replaced by $\UA_{1,n}(f_k) \smallsetminus\bigcup_\ell R_\ell$ for $\mathcal{R}$ as above.}

\ignore{ in particular, results in the following generalization of  Theorem \ref{thm: schleischitz extension}:
\begin{theorem}\label{thm: schleischitz extension general} Let   $n \geq 3$,  let $\mathcal{W}$ be a 
subset of $\R_{>0}^{n+1}$ {satisfying \equ{afw}},
and let $\{R_\ell\}$ be a countable collection of proper affine hyperplanes of $\R^n$.
Then for any   non-increasing $f:\R_{>0} \to \R_{>0}$ one has 
 $$\Big(\UA_{1,n}^*(f,\mathcal{W}) \smallsetminus\bigcup_\ell R_\ell\Big) + \Big(\UA^*_{1,n}(f,\mathcal{W}) \smallsetminus\bigcup_\ell R_\ell\Big)  = \R^n
  $$ and $$\Big(\bigcap_{\vw\in\R^{n+1}_{>0}}\UA^*_{1,n}(f,\vw)\Big) + \Big(\bigcap_{\vw\in\R^{n+1}_{>0}}\UA^*_{1,n}(f,\vw)\Big) = \R^n. $$
\end{theorem}}

  \begin{remark} \rm
It is not hard to show, by adapting the above proof, that for
any $n \geq 3$, 
any
positive non-increasing $f_1, \ldots, f_{n-1}$, and any $\mathbf{z}_1,
\ldots, \mathbf{z}_{n-1}$, we have
$$
\bigcap_{k=1}^{n-1}\left( \UA^*_{1,n}(f_k) - \mathbf{z}_k
\right) \neq \varnothing.
$$
 {Note that our method gives no information on the intersection of translates of  $\UA^*_{1,2}(f)$. In particular, it is an open problem to determine whether for any non-increasing $f:\R_{>0}\to \R_{>0}$ the sumset of $\UA^*_{1,2}(f)$ with itself coincides with $\R^2$.}
   \end{remark}
  \pagebreak
   


\section{Transference and improved rates for vectors\\ in {analytic submanifolds of $\R^n$}}\label{sec: transference} 

Given   {$\mathbf{x} = (x_1, \ldots, x_n)\in\R^n$} and {real numbers} $\tau, \vre, t, \eta$, define {the following} parallelepipeds:
$$
\Pi^{\tau,\varepsilon} \df \{(z_0,z_1,...,z_n) \in \mathbb{R}^{n+1}:\,\,\,
\max_{1\le j \le n} |z_j|\le \tau,\,\,\, |z_0 +z_1x_1+...+z_n x_n| \le \varepsilon
\}
$$
and
$$
\Pi_{t, \eta} \df 
\{(z_0,z_1,...,z_n) \in \mathbb{R}^{n+1}:\,\,\,
|z_0|\le t,\,\,\,
\max_{1\le j \le n} |z_0x_j - z_j|\le \eta
\}.
$$
Note that $\Pi^{\tau, \vre}$ encodes information about approximations to
$\mathbf{x}$ as a linear form, while $\Pi_{t, \eta}$ encodes approximations to
$\mathbf{x}$ as a vector. 

With these definitions,  we have the following standard transference result:
\begin{lemma}\label{lem: transference}
{\it  Let  $1\le g \le n$, suppose that $\tau, \vre, t, \eta$ satisfy 
\begin{equation}\label{equalities1}
  \frac{\eta}{t} = \frac{\varepsilon}{\tau} \end{equation}
and
\begin{equation}\label{equalities2}
\eta^{n-g}t =  {(4n)^{2n}} \tau^{g},
\end{equation} 
and {assume that} the parallelepiped $\Pi^{\tau,\varepsilon} $ contains $g$ {linearly} independent integer points.
Then $\Pi_{t, \eta}$ contains a {non-zero} integer point}.
\end{lemma}

\begin{proof}[Proof of Lemma \ref{lem: transference}]
Let  $ \mathbf{z}_1, \ldots ,\mathbf{z}_g$ be linearly independent integer points
in $\Pi^{\tau, \vre}$, let $$
{L} \df  {\rm span}\, (\mathbf{z}_1,\ldots,\mathbf{z}_g)$$ be the space
generated by these points, and denote the orthogonal complement of
${L}$ by ${L}^\perp$.  
Clearly  $
\Lambda \df {L}\cap\mathbb{Z}^{n+1}$ is a lattice in ${L}$,
and since ${L}^\perp$ is rational,  
$
\Lambda^\perp  = {L}^\perp\cap\mathbb{Z}^{n+1}$ is a lattice in
${L}^\perp.$ 
It is well known that the covolumes $\mathrm{covol}({L}/\Lambda)$
and  $\mathrm{covol}({L}^\perp/\Lambda^\perp)$  are equal to
each other.
Now define $ \Omega \df {L} \cap \Pi^{\tau, \vre}$ and  
$ \Omega^\perp \df {L}^\perp \cap \Pi_{t, \eta}$, and denote by 
 ${\rm vol}_g \Omega$  and ${\rm vol}_{n+1-g} \Omega^\perp$
 respectively, the $g$-dimensional and $(n-g+1)$-dimensional volumes
 of $\Omega$ and $\Omega^\perp$.

{Denote the vector $(1,x_1,\dots,x_n)\in\R^{n+1}$ by $\mathbf{x}_0$, and let 
$$
{M} =
\big\{(z_0,z_1,...,z_n) \in \mathbb{R}^{n+1}:\,\,\,
z_0 +z_1x_1+\cdots +z_nx_n = 0
\big\}
$$
be the hyperplane orthogonal to $\mathbf{x}_0$.
 Let us} define
$$\theta_{\max}
\df \max_{\mathbf{u} \in {L},\, \mathbf{v}\in {M}} \,(\text{angle between}\,\, \mathbf{u}\,\,\text{and}\,\, \mathbf{v})
$$ 
to be the maximal angle between vectors in ${L}$ and
${M}$, and 
$$
\theta_{\min} = \min_{\mathbf{u} \in {L}^\perp} (\text{angle between}\,\, \mathbf{u}\,\,
\text{and}\,\, \mathbf{x}_0).
$$
{It is easy to see that}
$\theta_{\min}\le  \theta_{\max}$.
Indeed, let $\mathbf{u}_0$ be the minimizer in the definition of
$\theta_{\min}$,  let $V$ be the two-dimensional subspace generated by $\mathbf{x}_0$
and $\mathbf{u}_0$, and let  $\mathbf{u}_0^\perp$
and $\mathbf{x}_0^\perp$ be  unit vectors in $V$ {perpendicular to $\mathbf{u}_0$ and $\mathbf{x}_0$ respectively}. Then 
$$
\theta_{\min} = \text{ angle } (\mathbf{u}_0, \mathbf{x}_0) = \text{
  angle } (\mathbf{u}_0^\perp, \mathbf{x}_0^\perp) \leq \theta_{\max}.
$$
{Let us now consider two cases. Suppose that}
$\sin \theta_{\min} \le \frac{\eta}{t}$.
Then we have a lower bound
\begin{equation}\label{a121}
 {\rm vol}_{n-g+1} \Omega^\perp \ge
 {t \eta^{n-g}}
 ,
\end{equation}
and an upper bound
\begin{equation}\label{a1010}
 {\rm vol}_g \Omega \le 2{n}^{g/2} \tau^g
 .
\end{equation}
Note that $\Omega$ contains the convex hull 
of {$\{\pm \mathbf{z}_1,
\ldots, \pm \mathbf{z}_g\}$},  and the
{parallelepiped
$$\left\{\mathbf{z} = \lambda_1 \mathbf{z}_1+ \ldots+
\lambda_g\mathbf{z}_g,\,\, \max_{1\le i\le g}|\lambda_i|\le \frac{1}{2}\right\}
\subset \frac{g}{2} \cdot \Omega
$$
contains a fundamental domain for $\Lambda$.}
Therefore we
have 
$$
\begin{aligned}
 \mathrm{covol}({L}/\Lambda) &= \mathrm{covol}({L}^\perp/\Lambda^\perp)
\le
{\frac{g^g\cdot {\rm vol}_g  \Omega}{2^g}} \stackrel{\eqref{a1010}}{\le}
{\frac{g^g n^{g/2}\tau^g}{2^{g-1}}}\\
&{
\stackrel{\eqref{equalities2}}{=} 
\frac{g^gn^{g/2-2n}}{2^{4n+g-1}} \cdot t \eta^{n-g}} \le
\frac{1}{2^{n-g+1}} \cdot t \eta^{n-g}
\stackrel{\eqref{a121}}{\le}   \frac{{\rm vol}_{n-g+1} \Omega^\perp}{2^{n-g+1}}.
\end{aligned}
 $$
By  the Minkowski Convex Body Theorem there exists a {non-zero} integer point in
$\Omega^\perp \subset \Pi_{t, \eta}$. {Hence} we are done in this case. 

Now suppose {that}
$\sin \theta_{\min} > \frac{\eta}{t}$.  Then 
\begin{equation}\label{a12}
 {\rm vol}_{n-g+1} \Omega^\perp \ge
 \frac{\eta^{n-g+1}}{\sin \theta_{\min}}
\end{equation}
and
\begin{equation}\label{a101}
 {\rm vol}_g \Omega \le {n}^{g/2} \tau^{g-1} \cdot \frac{\varepsilon}{\sin  \theta_{\max}}.
\end{equation}
Arguing as in the previous case {but} using
\eqref{a12} and \eqref{a101} instead of \eqref{a121} and \eqref{a1010}, 
we {obtain} 
 \[\begin{split}
\mathrm{covol}({L}/\Lambda) = \mathrm{covol}({L}^\perp/\Lambda^\perp)&\le
\frac{g!}{2^g} \cdot {\rm vol}_g \Omega  \le  \frac{g!
  n^{g/2}}{2^{g}}\cdot \frac{\tau^{g-1} \varepsilon}{ \sin
  \theta_{\max}} \\ &\le  \frac{\eta}{t}\cdot \frac{\eta^{n-g}t}{2^{n-g+1}\sin\theta_{\min}}\le
\frac{{\rm vol}_{n-g+1} \Omega^\perp}{2^{n-g+1}},
\end{split}\]
and, again  by  the Minkowski Convex Body Theorem,
we get a
 {non-zero} integer point in
$\Omega^\perp \subset \Pi_{t, \eta}$.
  \end{proof}
  
  {Now let us state and prove a generalization of Theorem \ref{thm: improvement rate}.
  \begin{theorem} \label{cor: vectors manifolds}
	Let $n\ge 2$, let $Y$ be a $d$-dimensional connected analytic submanifold of $\R^n\cong M_{n,1}(\R)$, where $d\ge 2$, and let $\{R_\ell \}$ be a   countable  
        collection  of proper {closed} analytic submanifolds of $Y$. 
        Suppose that  a 
        {continuous non-decreasing}
        function $f$ satisfies
\eqref{eq: goes to infty monotonically }.
        Then the set
       $${Y\cap \UA_{n,1}(f)  \smallsetminus \bigcup_{\ell\in\N} R_\ell}$$
	is uncountable and dense in $Y$. In particular, if one in addition assumes that $Y$ is 
	not contained in any proper rational affine subspace of $\R^n$, then the set
        $Y\cap \UA^*_{n,1}(f)
       $ is uncountable and dense in $Y$.
\end{theorem}}

  \begin{proof}
  {
 For $t> 0$  and  $g \df d-1$, let us  set $\eta \df f (t)$ and $\tau \df  \left(
  \frac{\eta^{n-g}t}{ {(4n)^{2n}}} \right)^{1/g},  $ so that
\eqref{equalities2} is satisfied,
 {and the function $t\mapsto \tau$ is  continuous.}
By \eqref{eq: goes to infty monotonically }, we have that 
$ \tau \to \infty$ monotonically when $t\to \infty$.
The map $t \mapsto \tau$ is thus bijective, 
  hence one can consider the inverse map $ \tau\mapsto t(\tau)$  and define a positive function  
$$\varepsilon = {h}(\tau) \df \frac{\tau f \big(t(\tau)\big)}{t(\tau)},
$$  so that \eqref{equalities1} holds. Note that $h$ is non-increasing, since so is the function 
$$
h\big(\tau(t)\big) = \frac{\left(
  \frac{f(t)^{n-g}t}{ {(4n)^{2n}}} \right)^{1/g}f(t)}t = \frac{f(t)^{n/g}}{ {(4n)}^{2n/g}t^{1-1/g}}.$$
We will
prove 
the implication
\eq{implication}{\mathbf{x}^T\in \UA_{1,n}(h;g) \ \Longrightarrow\ 
\mathbf{x}\in \UA_{n,1}(f).}
 This, in view of Theorem \ref{cor: linear forms manifolds}, will immediately imply the conclusion of Theorem \ref{cor: vectors manifolds}.}

{To prove \equ{implication}, note that 
$\mathbf{x}^T\in \UA_{1,n}(h;g)$ amounts to saying that for any large
enough  $\tau$ the parallelepiped
$\Pi^{\tau,\varepsilon}$ contains $g$ linearly independent integer points. In view of
Lemma \ref{lem: transference} we get that for all $t$ large
enough,  the parallelepiped $\Pi_{t, \eta}$ contains a nonzero integer point $\mathbf{z} =
(q, \mathbf{p})$. For $t$
large enough, the first coordinate $q$ will be nonzero;
thus we can find $q \in \Z\nz$ and $\mathbf{p} \in \Z^n$ such that  
$q \leq t$ and $\|q\mathbf{x} -\mathbf{ p} \| \le f(t).$}
\end{proof}

\section{Further applications
}\label{further}

{Most of the results described in this section are not proved in this paper; the proofs will appear elsewhere.}

\subsection{Irrationality measure functions and the Kan--Moshchevitin phenomenon} \label{kanmos}
Another convenient way to describe various results in the theory of \da\ is through   irrationality measure functions. Namely,  given
\amr, one defines its {\sl irrationality measure function} by
$$
\psi_{A}: \R_{>0} \to \R_{>0}, \ \ \psi_{A}(t) \df 
\inf \big\{\|A\mathbf{q}  - \vp\| : \mathbf{q} \in \Z^n\nz, \ \|\mathbf{q} \| \leq t, \ \vp \in \Z^m \big\}. 
$$
Then 
it is easy to see that    $A$ is $f$-uniform\footnote{Likewise one can also use the function $
\psi_{A}$ to study {\em asymptotic approximation}:  requiring that
$\psi_{A}(t) \le   f(t)$ for an unbounded set of $t>0$, one gets a definition of {\em $f$-approximable} systems of linear forms $A$.}
  if and only if $\psi_{A}(t) \le  f(t)$ for all
large enough $t$. Similarly 
 for arbitrary Diophantine system $\mathcal{X} =
     \left( X, \mathcal{D}, \mathcal{H}\right)$  one can introduce the   {irrationality
  measure function} associated to $x\in X$:
$$
\psi_x(t) \df  \inf \{ d_s(x):  h_s\leq t\};
$$    
then    $x\in \UA_{\mathcal{X}}(f)$  if and only if $\psi_x(t) \le f(t)$ for all large enough $t$.
\ignore{Our main theorem (Theorem \ref{thm: special axiomatic theorem}) thus claims that under certain conditions on $ \mathcal{X}$, for any positive non-increasing function $f$ 
the set  of $ x\in X$ such that 
$$
\exists \, t_0\,\,\, \forall\, t\ge t_0\,\,\,\, \psi_x(t)< f(t)
$$
is at least non-empty.}

In 2009 the following result was proved by Kan and Moshchevitin \cite{KaMo}: in the standard Diophantine system corresponding to approximations to one real number ($m=n=1$), for any two different real numbers  $x,y$ with 
 $$
\psi_x(t) >0,\,\,\,\,
\psi_y(t) >0 \,\,\,\forall \, t
$$
(a condition equivalent to $x,y\notin\Q$),
the difference
 $
\psi_x(t) -
\psi_y(t) 
$
changes its sign infinitely often as $t \to \infty$. This phenomenon is not present when $\max(m,n) > 1$. And more generally, as long as the assumptions of our main theorem (Theorem \ref{thm: special axiomatic theorem}) are satisfied, there exists pairs of points whose irrationality
  measure functions do not exhibit the pattern of infinitely many changes of signs. More precisely, suppose   a  Diophantine system  $${\mathcal{X} = \big( X, \mathcal{D} = \{d_{s}: s\in\I\}, \mathcal{H} = \{h_{s}: s\in\I\}\big)}$$
	and a countable 
        collection $\mathcal{L}$ of closed   subsets of $X$ 
       satisfy the assumptions of Theorem \ref{thm: special axiomatic theorem} with $Y = X$, $\varphi_k\equiv \Id$ and  $\mathcal{R} = \{{d_s}^{-1}(0) : s\in \I\}$. Then for any $x\in X$ such that $$d_{s}(x)\ne 0\quad\Longleftrightarrow  \quad \psi_x(t) >0 \,\,\,\forall \, t$$
one can apply the theorem with $f = \psi_x$ and conclude that there exists a dense set of $y\in X$ such that 
 $
\psi_y(t) >0$  for all $t$ and $\psi_y(t)< \psi_x(t)$ for all   large enough $t$.

\subsection{Inhomogeneous approximation and approximation with restrictions} \label{inh}
It is natural to consider the problems of uniform approximation for systems of linear forms in the inhomogeneous set-up; that is, fix $\b\in\R^m$ and, instead of \equ{digeneral},  look for nontrivial integer solutions of the system
\eq{digeneralinh}{
{\|A\vq + \b - \vp \|
\le  f(t)}
  \ \ \ \mathrm{and}  
\ \ \|\vq\|
\le  t
.}  
This calls for considering 
the collection \eq{ellinh}{\begin{aligned}\{L_{\vp,\vq,\b} : \vq \in\Z^n\nz, \ \vp\in\Z^m\},\\
 \text{  
where }L_{\vp,\vq,\b} \df \{A\in\mr: A\vq + \b= &\vp\},
\end{aligned}}
and proving its total density for arbitrary $\b\in\R^n$. And indeed it turns out to be  possible to achieve this when $n > 1$, and thereby construct $f$-uniform  systems of affine forms with an arbitrary fixed translation part. Moreover, in a forthcoming joint work with {Hong and} Neckrasov we 
 take subsets $P$ of  $\R^m$ and  $Q$ of $\Z^n$, and say that 
 \amr\ is  \textsl{$f$-uniform with respect to $(P,Q)$}
if 
for all
large enough $t>0$  there exists
$\vq  \in Q$ and $\vp \in P$ satisfying  \equ{digeneral}. The classical theory of homogeneous approximation corresponds to $P = \Z^m$ and $Q = \Z^{{n}}\nz$. 
The following can be proved (work in progress):

\begin{theorem} \label{vasya} The collection ${{\{L_{\vp,\vq}}: \vp\in P,\,\vq\in Q\}}$ of subspaces of $\mr$ is totally dense if

 \smallskip
 \begin{itemize}
\item[\rm(a)]  $Q = \Z^n\nz$ and  $P$ is a subgroup of $\Z^m$ of rank $>m-n+1$ ({restricted numerators});

 \smallskip\item[\rm(b)]  $Q = Q_1\times\cdots\times Q_n\subset \Z\times\cdots\times\Z$, 
where
 at least two of the sets $Q_i$ are infinite  ({restricted denominators}), 
and $P$ is of bounded Hausdorff distance from $\R^m$.
\end{itemize}
Consequently (modulo Theorem \ref{thm: special axiomatic theorem})  for any 
        non-increasing 
$f: \R_{>0} \to \R_{>0}$ the set of 
     $m \times n$ matrices	 that are  $f$-uniform with respect to $(P,Q)$ and not contained in a given countable family of proper analytic submanifolds of $\mr$ is
     uncountable and dense. 
 \end{theorem}

Note that both (a) and (b) implicitly assume that $n > 1$, and the set-up of (b) includes inhomogeneous approximation by letting $P = \Z^m+\b$ {for a fixed $\b\in\R^m$}.
{When $n=1$, it is clear that for any fixed $\b\in\R^m$ the collection \equ{ellinh} is not totally dense.
On the other hand one can study a doubly metric version of the problem, that is, the set of 
pairs $(A,\b$) such that \equ{digeneralinh} has a non-trivial integer solution for large enough $t$.
It was recently 
{documented}
by the second named author \cite{Mnew} that for any $m\in\N$ and any   non-increasing 
$f: \R_{>0} \to \R_{>0}$  there exists a dense and uncountable set of pairs $(\x,\y)\in\R^m\times \R^m \cong M_{m,2}$ such that the system of inequalities
$$\|q\x + \y - \vp\| \le  f(t)
  \ \ \ \mathrm{and}  
\ \ |q|
\le  t 
$$
has a non-zero integer solution $(\vp,q)$ for all large enough $t$. 
{In fact,  this statement follows from a old result by Khintchine \cite{Khi}
which is not very well known. See also \cite[\S 3.3]{MN} for a discussion.}
\ignore{is totally dense in $\R^m \times \R^m\
\|A\vq + \b - \vp \|
\le  f(t)
  \ \ \ \mathrm{and}  
\ \ \|\vq\|
\le  t
 by allowing $\b$ to vary. Namely, one can prove that the collection
{$$\big\{\{(\x,\b)\in\R^m\times \R^m: q\x + \b= \vp\}:q \in\Z\nz, \ \vp\in\Z^m \big\}
$$
is totally dense in $\R^m \times \R^m\cong M_{m,2}$. This can be done by a modification of the proof of Theorem \ref{vasya}(b). As a consequence, one }
Clearly the same result can be established in the set-up of weighted \da\ as in \equ{weights}. Note that the case $m=n=1$ has been proved earlier by the second-named author using continued fractions (unpublished).}

\subsection{Rational approximations to linear subspaces} \label{grassm}
In this subsection we fix $d\in\N_{\ge2}$ and $a = 1,\dots, d-1$, and consider a certain Diophantine system on $X = \Gr_{d,a}$, the Grassmannian of $a$-dimensional  subspaces of $\R^d$. Following the set-up of \cite{Schmidt}, one can study the problems of approximation of  an $a$-dimensional  subspace $A$ of $\R^d$ by $b$-dimensional rational subspaces $B$ of $\R^d$, where $1\le b < d$, in terms of the so-called {\sl first angle} between the subspaces. 
The latter (or rather, formally speaking, the sine of the angle) is defined as follows:  
$$
\measuredangle_1 (A,B) \df  \min_{\x \in A\nz,\, \y \in B\nz}
\frac{\|\x\wedge \y\|}{\|\x\|\cdot \|\y\|},
$$
where $\|\cdot\|$ is the Euclidean norm on $\R^d$ and on $\bigwedge^2(\R^d)$.

 Note that $
\measuredangle_1 (A,B) = 0$ if and only if $\dim(A\cap B) >0$. Using the terminology introduced in \cite{N}, let us say that  $A\in \Gr_{d,a}$ is {\sl completely irrational} if for any $(d-a)$-dimensional rational subspace $R$ we have
$A\cap R = \{{0}\}$. We also define the {\sl height} $H(B)$ of a rational subspace $B\subset \R^d$ { of dimension $b$ in the natural way as the
covolume of the $b$-dimensional lattice $ B\cap \mathbb{Z}^d$}. Using the methods of this paper, namely Theorem \ref{thm: special axiomatic theorem}, it is possible to prove the following

\begin{theorem} \label{angles} Let $1\le a,b < d$ with $\max(a,b) > 1$. 
 Then for any 
        non-increasing 
$f: \R_{>0} \to \R_{>0}$ the set of 
      completely irrational	$A\in \Gr_{d,a}$ 
such that   the system of inequalities
$$
H(B)  \le t,\,\,\, \measuredangle_1 (A,B) \le f (t)
$$
has a solution 
in $b$-dimensional rational subspaces $B$  for all  large enough $t$  is
     uncountable and dense. 
 \end{theorem}

The case $d=4$ and $a=b=2$ is a recent result of Chebotarenko \cite{Ch}. The above theorem, as well as several generalizations, will be proved in a forthcoming work.

\end{document}